\documentclass[a4paper]{amsart}

\usepackage{listings}

\usepackage{amssymb}
\usepackage{amsthm}  
\usepackage{amsmath} 
\usepackage{hyperref}
\usepackage{color}
\usepackage{graphicx}
\usepackage{subfigure,algorithm}
\hypersetup{linktocpage}

\newcommand{\va}{\varphi}
\newcommand{\fg}{\mathfrak g}

\newcommand{\fp}{\mathfrak p}

\newcommand{\fk}{\mathfrak k}

\newcommand{\fm}{\mathfrak m}

\newcommand{\Fl}{{\rm Fl}}

\newcommand{\R}{\mathbb{R}}

\newcommand{\ad}{{\rm ad}}

\newcommand{\D}{\mathcal{D}}
\newcommand{\del}{\partial}
\newcommand{\HH}{\mathcal{H}}

\newcommand{\N}{\mathcal{N}}

\theoremstyle{plain}

\theoremstyle{plain}
\newtheorem{thm*}{Theorem}[section]
\newtheorem{prop*}[thm*]{Proposition}
\newtheorem{lem*}[thm*]{Lemma}
\newtheorem{cor*}[thm*]{Corollary}

\newtheorem{rem}[thm*]{Remark}

\theoremstyle{definition}
\newtheorem{def*}{Definition}
\newtheorem{example}{Example}

\theoremstyle{remark}

\begin{document}

\title{Geometric control theory of vertical rolling disc using symmetries}

\author{Jaroslav Hrdina\,$^a$, Ale\v s N\'avrat\,$^{a,d}$, Lenka Zalabov\'a\,$^{b,c}$}

\address{
	$^a$\,
	Institute of Mathematics, 
	Faculty of Mechanical Engineering,  Brno University of Technology,
	Technick\' a 2896/2, 616 69 Brno, Czech Republic
	\\	
	$^b$\,
	Institute of Mathematics, 
	Faculty of Science, University of South Bohemia, 
	Brani\v sovsk\' a 1760, \v Cesk\' e Bud\v ejovice, 370 05, Czech Republic 
	\\
	$^c$\,
	Department of Mathematics and Statistics, 
	Faculty of Science, Masaryk University,
	Kotl\' a\v rsk\' a 2, Brno, 611 37, Czech Republic 
    \\
	$^d$\,
	International School for Advanced Studies, 34136 Trieste, Italy}	 
\email{hrdina@fme.vutbr.cz, navrat.a@fme.vutbr.cz, lzalabova@gmail.com}

\keywords{vertical rolling disk, Lagrangian contact geometry, local control, sub--Riemannian geometry, Pontryagin's maximum principle, nilpotent Lie group}
\subjclass[2010]{53C17, 70Q05, 22E60, 37J60}

\maketitle              

\begin{abstract}
We use the methods of geometric control theory to study extremal trajectories of vertical rolling disk. We focus on the role of symmetries of the underlying geometric structures. We demonstrate the computations in the CAS Maple package DifferentialGeometry.
\end{abstract}

In this article we apply methods of geometric control theory on the vertical rolling disc, \cite{bloch2015}. We use the Hamiltonian viewpoint and Pontryagin's maximum principle to discuss corresponding optimal control problems, \cite{AS,ABB}, and focus on the role of symmetries of the control system and related geometric structures, \cite{parabook}. In Section \ref{sec_ex} we describe the mechanical system of the vertical rolling disc. We also introduce some other mechanisms related to the disc. We also stu\text{d}y controllability of the systems.
 
In Section \ref{akce-g-minus} we formulate the control problem for the optimal movement of the vertical rolling disc. We show that the problem can be studied as a control problem on a Lie group and in this way, we formulate the corresponding system for local extremals. It turns out that solutions of the system cannot be described easily.
In Section \ref{sec_nilp} we describe homogeneous nilpotent approximation of the system, \cite{as87,bel,h86}. We show that the approximation is modelled on the Heisenberg algebra. Corresponding control system is a control system on a nilpotent Lie group that approximates the original system and can be solved explicitly.  We compare the solutions with numerical solutions of the original system on examples.
 
In Section \ref{nilp_sym} we focus on the role of the symmetries of the system. We find symmetries of the nilpotent control system and we study the action of symmetries on local minimizers. In particular, we use isotropy symmetries to discuss conjugate locus.
In Section \ref{sec_metr} we stu\text{d}y possible choices of the control metric. In particular, we focus on the Lagrangian contact structure \cite{parabook} which is determined by  two distinguished directions in the configuration space that play the roles of plane a angular velocities and give a distinguished class of metrics. We also stu\text{d}y its symmetries.
Let us note that role of the geometric structure of contact manifolds are also studied, s.e.g. \cite{en,en2}. 

In Appendix \ref{tanaka} we give an explicit computation of algebraic and geometric Tanaka prolongation which gives symmetries of the Lagrangian contact structure in question, \cite{yam,zel}. We use the computer system Maple and the package DifferentialGeometry to realize the computation and we present it in the text and in Appendix \ref{sec_mapl}, \cite{dg}.

\section{Introduction to plane mechanics} \label{sec_ex}
We introduce several simple plane mechanisms with configurations spaces related to Heisenberg structure. In particularly, we discuss control Lie algebras of the systems. We also mention relations  between vertical rolling disc and kinematic car. 

\subsection{Vertical rolling disc} \label{definice}
We stu\text{d}y a vertical disc rolling in the plane, which is the classical problem that often appears as an example in geometric mechanics,  \cite{bloch2015,ma}.
 Classical references for the vertical rolling disk are \cite{her1995,rei1996}
 and we follow conventions introduced therein. The configuration space of the disc is
$
 \mathbb R^2 \times \mathbb S^1 \times \mathbb S^1
$ with coordinates 
$q = (x,y, \theta, \varphi)$, where $(x,y)$ is the position of the contact point $P_0$  in the plane $\mathbb{R}^2$, $\theta$ is the orientation of the
disk in the plane, and $\varphi$ is the rotation angle of the disk with respect to a fixed $P$, see Figure \ref{disk}. 
\begin{figure}[h]
\includegraphics[width=0.34 \textwidth]{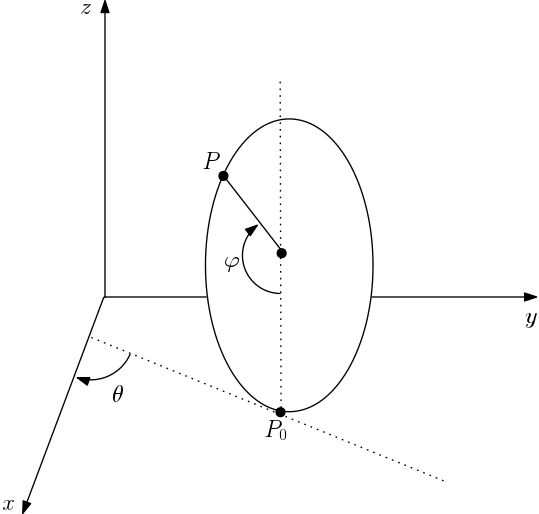}  
\caption{Vertical rolling disc} \label{disk}
\end{figure}
We assume that the radius of the disc equals to $1$. We assume that the rolling of the disc is without slipping nor sliding. Im other words, we suppose that the direct plane velocity of the disc is proportional to angular velocity of the circular motion.
In coordinates, the nonholonomic constraints are
\begin{align} \label{sys_1}
\begin{split}
\dot x & =  \cos \theta \dot \varphi, \\
\dot y & =  \sin \theta \dot \varphi.
\end{split}
\end{align}
We can reformulate equations \eqref{sys_1} in the formalism of differential forms as a Pfaff system     
\begin{align} \label{sys_2}
\begin{split}
\text{d}x- \cos \theta  \text{d} \varphi &=0,  \\ 
\text{d}y- \sin \theta  \text{d} \varphi &=0.  
\end{split}
\end{align}

The annihilator of the Pfaff system \eqref{sys_2} forms the distribution which is spanned by vector fields 
\begin{align} 
\label{f2}
\begin{split} 
Y_1 &= \partial_{\theta}, \\
Y_2 & = \cos \theta  \partial_x + \sin \theta \partial_y + \partial_{\varphi}.  
\end{split}
\end{align} 
In fact, one can read of the solution of \eqref{sys_1} by the choice of two  parameters as $\dot \theta$ and $\dot \varphi $.
To study the kinematic of controlled rolling disk we reformulate the system \eqref{sys_1} as the first--order control system  
\begin{equation} \label{con_sys}
\dot q = u_1 Y_1 + u_2 Y_2,
\end{equation}
where 
$\dot q = \begin{pmatrix} \dot x, \dot y,\dot  \theta, \dot \varphi  \end{pmatrix}$. 
The parameter $\va$ is necessary for the analysis of \text{d}ynamics, \cite{bloch2015}.
There is no need to use the parameter $\va$ for our considerations in control. 
To eliminate $\varphi$ from the system \eqref{sys_1},  we write 
\begin{align} \label{sys_3}
\dot y \cos \theta = \dot x \sin \theta.
\end{align}
We can reformulate the system \eqref{sys_3} as Pfaff system     
\begin{align} \label{sys_4}
\begin{split}
\text{d}y \cos \theta    -  \text{d} x \sin \theta 
=0,
\end{split}
\end{align}
and the annihilator of the system \eqref{sys_4} forms the distribution $\D$ which is spanned by vector fields  
\begin{align}
\label{fields} 
\begin{split}  
X_1 & =\partial_{\theta}, \\
X_2 &=  \cos \theta \partial_x +  \sin \theta \partial_y.
\end{split} 
\end{align}
Then we can rewrite the system \eqref{sys_4} as the control system  
\begin{equation} \label{con_sys_2}
\dot q = u_1 X_1 + u_2 X_2,
\end{equation}
 where $\dot q = \begin{pmatrix} \dot x, \dot y,\dot  \theta \end{pmatrix}$ for 
 $\begin{pmatrix}  x,  y, \theta \end{pmatrix} \in \mathbb R^2 \times \mathbb S^1$.

In the Listing \ref{code_pf} we introduce differential geometry packages, define the configuration space, the Pfaff form \eqref{sys_4} and find its annihilator in Maple. 

\tiny 
\begin{lstlisting}[language=Mathematica ,caption={Pfaff system}
, label={code_pf} ]
restart: with(DifferentialGeometry): with(Tools): with(LieAlgebras):
with(Tensor): with(PDEtools, casesplit, declare): with(GroupActions):
with(LinearAlgebra):
DGsetup([x, y, theta], M);
pf := evalDG(dy*cos(theta) - dx*sin(theta));
an := Annihilator([pf]);
X1 := an[1]; X2 := an[2];
# Different version of Maple can give the different solutions, the following
# modification is necessery in Maple 2019 to get the correct generators.
X1 := an[1]; X2 := evalDG(sin(theta)*an[2]);
\end{lstlisting}
\normalsize

\subsection{Controllability of the mechanisms} 
\label{control}
The systems \eqref{con_sys} and \eqref{con_sys_2} 
  form control systems over configuration spaces $\mathbb R^2 \times \mathbb S^1 \times \mathbb S^1$ and  $\mathbb R^2 \times \mathbb S^1$, respectively, with control functions $u_1$ and $u_2$. The controllability of each linear system
    means that for arbitrary fixed points $x_0,x_1 $ of the configuration space, there exists an admissible control that steers the system from 
      $x_0$ to $x_1$ in the finite time.  
Controllability  of symmetric affine systems  is  completely characterized by controllability Lie algebra by Chow--Rashevskii theorem,  \cite{J,S,MZS}.
\begin{prop*}The control system  \eqref{con_sys} is controllable everywhere in the configuration space  $\mathbb R^2 \times \mathbb S^1 \times  \mathbb S^1$. 
\end{prop*}
\begin{proof} 
The controllability Lie algebra is spanned by vector fields \eqref{f2}, i.e.
$ Y_1= \partial_{\theta} $  and 
$Y_2= \cos \theta  \partial_x + \sin \theta \partial_y + \partial_{\varphi}$.  
Let us take Lie brackets  
\begin{align*}
Y_{12}&:=[Y_1,Y_2] =  -\sin \theta \partial_x + \cos \theta \partial_y  ,\\ 
Y_{112}&:=[Y_1,Y_{12}] =  -\cos \theta  \partial_x - \sin \theta \partial_y. 
\end{align*}
The determinant of the Jacobi (control) matrix 
\begin{align} \label{jacobi}
\left (
\begin{smallmatrix}
0&0 & 1 & 0\\
\cos \theta  & \sin \theta  & 0& 1 \\
-\sin \theta  & \cos \theta & 0 &0 \\  
 -\cos \theta  & - \sin \theta & 0 & 0  
 \end{smallmatrix}
\right )
\end{align}
is equal to $1$ in any point of the configuration space. Thus the controllability rank condition holds in any point of the configuration space and the system is controllable in any points by Chow--Rashevskii theorem.
\end{proof}
\begin{cor*}
The control system  \eqref{con_sys_2} is controllable everywhere in the configuration space 
$\mathbb R^2 \times \mathbb S^1$.
\end{cor*}
\begin{proof} 
The controllability Lie algebra is generated by vector fields \eqref{fields}, i.e.
$ X_1= \partial_{\theta} $ and
$X_2= \cos\theta \partial_x + \sin\theta\partial_y $.  
Let us consider their Lie bracket
\begin{align*}
X_{12}&:=[X_1,X_2] =  -\sin\theta \partial_x + \cos\theta\partial_y,
\end{align*}
which is exactly the Reeb field of \eqref{sys_3}. The determinant of corresponding Jacobi (control) matrix is a minor in \eqref{jacobi} for fourth row and fourth column, which is equal to $1$ in any point of $M$. Thus the controllability rank condition holds in any point of $M$ and the system is controllable in any points by Chow--Rashevskii theorem.
\end{proof}

Let us note that the vector fields  together with their non--trivial brackets determine a $4$--dimensional and $3$--dimensional, respectively, non-nilpotent solvable 
Lie algebras, see Table \ref{tab}.
\begin{table}[h] 
\begin{center}
\begin{tabular}{ c|| c | c | c | c|} 
 & $Y_1$ & $Y_2$ & $Y_{12}$ & $Y_{212}$ \\
\hline \hline
$Y_1$ & $0$ & $Y_{12}$ & $Y_{112}$ & $Y_2$  \\ 
\hline
 $Y_2$ & $-Y_{12}$ & $0$ & $0$ & $0$  \\
\hline
$Y_{12}$ & $-Y_{112}$ & $0$ & $0$ & $0$  \\
\hline
$Y_{112}$ & $-Y_2$ & $0$ & $0$ & $0$  \\
\end{tabular}  \: \: \: \: \: \: 
\begin{tabular}{ c|| c | c | c |}
 & $X_1$ & $X_2$ & $X_{12}$  \\
\hline \hline
$X_1$ & $0$ & $X_{12}$ & $-X_2$  \\ 
\hline
$X_2$ & $-X_{12}$ & $0$ & $0$   \\
\hline
$X_{12}$ & $X_2$ & $0$ & $0$   \\
\end{tabular}
\end{center}
\caption{Controllability algebras of the systems \eqref{con_sys}  and \eqref{con_sys_2}
}
\label{tab}
\end{table}

In the Listing \ref{cod_cont} we compute the Lie bracket $X_{12}:=[X_1,X_2]$, introduce the Lie algebra $\langle X_1,X_2, X_{12} \rangle $ and discus its properties. 
\tiny 
\begin{lstlisting}[language=Mathematica ,caption={Controllability Lie algebra}
, label={cod_cont} ]
X12 := LieBracket(X1, X2);
Alg := LieAlgebraData([X1, X2, X12]);
DGsetup(Alg);
MultiplicationTable();
Query("Solvable");
Query("Nilpotent");
\end{lstlisting}
\normalsize

\subsection{Relation to Dubin's car} 
Let us now briefly focus on the kinematic car, \cite{ma,MZS,mich}.  
The configuration space of the car is $  \mathbb R^2 \times \mathbb S^1 \times \mathbb S^1 $ with coordinates $q = (x,y, \theta, \varphi)$, where $(x,y)$ is the position of the center of the rear pair of wheels, $\theta$ is the orientation of the car in the plane, and $\varphi$ is the rotation angle of the 
front pair of wheels, see Figure \ref{car}.

\begin{figure}[h]
\includegraphics[width=0.4\textwidth]{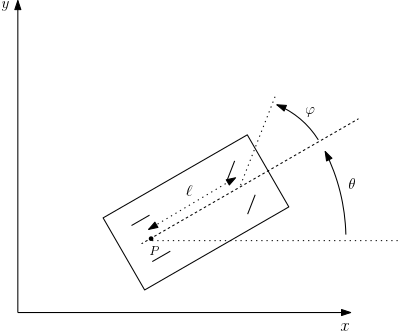} 
\caption{Kinematic car} \label{car}
\end{figure}

The Pfaffian constraints on the admissible movements of the car are  
\begin{align} \label{car 2}
\begin{split}
\sin( \theta + \varphi) \dot x - \cos (\theta + \varphi)
\dot y - \ell \cos \varphi \dot \theta &= 0,\\
\sin\theta\dot x - \cos\theta\dot y & =0.
\end{split}
\end{align} 
The corresponding control system  
 with the inputs chosen as the plane velocity $u_1=\dot \va$ and the angular velocity $u_2$ has the form $\dot q=u_1Z_1+u_2Z_2$ for
\begin{align*} 
 Z_1 &= \partial_{\va}, \\
Z_2 &=  
 \cos\theta\partial_x + \sin\theta\partial_y + \frac{\tan \varphi}{\ell}   \partial_{\theta}.
 \end{align*}
We would like to discuss the relation of this system with the system \eqref{con_sys_2}.  Let us avoid the role of the rear pair of wheels and fix the orientation of the front pair of wheels.
This means that we ignore the second nonholonomic constraint of \eqref{car 2} and we fix $\varphi \in \mathbb S^1$. 
This leads to the control system for  
\begin{align}
\begin{split} 
\bar Z_1 &= \cos\theta \partial_x + \sin\theta \partial_y, \\
\bar Z_2 &= -\sin\theta \partial_x + \cos\theta  \partial_y - \frac{1}{\ell} \partial_{\theta}.
\end{split}  \label{vert_rol}
\end{align}
These two fields together with the bracket
\begin{align*}
\bar Z_{12} &= [\bar Z_1,\bar Z_2] =  -\frac{\sin\theta}{\ell} \partial_x +
  \frac{\cos\theta}{\ell} \partial_y
\end{align*}
form the  three--dimensional solvable Lie algebra with the multiplicative table \ref{tab}.

\begin{rem}
 The term Dubin's path typically refers to the shortest curve that connects two points in the two-dimensional Euclidean plane and satisfies a constraint on the curvature. It turns out that constraints determining Dubin's paths are equivalent to the constraints for the movement of vertical rolling disc,  \cite{dub,reeds}. 
\end{rem}

\section{Control of the vertical rolling disk} \label{akce-g-minus}
In the next, we will discuss only the vertical rolling disc \eqref{fields}. One can adapt the methods for related mechanisms and particularly for mechanisms of kinematic car corresponding to \eqref{vert_rol}, that has the same structure of the configurations space.

\subsection{Formulation of the problem} \label{pmp-orig}
The configuration space  carries a canonical filtration with the growth $(2,3)$ given by the fact that the disc can move only in the direction of the horizontal distribution $\D$, which is generated by the fields $X_1=\del_\theta$ and $X_2=\cos\theta\del_x+\sin\theta\del_y$. Moreover, we define the sub--Riemannian metric $k$ in $\D=\langle X_1,X_2\rangle$ such that the fields $X_1$, $X_2$ are orthonormal with respect to $k$. Explicitly, $$k= \text{d}\theta^2 + (\cos\theta \text{d}x+ \sin\theta \text{d}y)^2
.$$

In the Listing \ref{cod_met} we compute the dual basis of the basis $X_1,X_2,X_{12}$ and define the control metric $k$.
\tiny 
\begin{lstlisting}[language=Mathematica ,caption={sub--Riemannian metric}, label ={cod_met}]
ChangeFrame(M);
db := DualBasis([X1, X2, X12]);
k := evalDG((db[1] &t db[1]) + (db[2] &t db[2]));
\end{lstlisting}
\normalsize
The fields $X_1,X_2$ and $X_{12}$ determine the Lie algebra
$$
\fk=(\langle X_1,X_2,X_{12} \rangle,[X_1,X_2]=X_{12},\ [X_1,X_{12}]=-X_2),
$$
see Table \ref{tab}.
In particular, we can view locally the configuration space as the connected, simply connected Lie group $K$ of $\fk$ and the fields $X_1,X_2,X_{12}$ are left-invariant with respect to the group structure on $K$.
Altogether, there is a Lie group structure such that the sub--Riemannian structure $(K,\D,m)$ is left-invariant with respect to this group structure.

Consider two configurations $k_1, k_2 \in K$. Among all admissible curves $c(t)$, i.e. locally Lipschitz
curves such that $c(0) = k_1$ and $c(T) = k_2$ that are tangent to $\D $ for almost all $t \in [0, T]$, we would like to find length minimizers with respect to control metric $k$. We would like to minimize the length $l$ among all the horizontal curves $c$, where the length is given by $l(c)=\int_0^T\sqrt{k(\dot c, \dot c)} dt$ for  $k$. Let us recall that the distance between two points $k_1,k_2 \in K$ is defined as $d: K \times K \to [ 0, \infty]$,
$d (k_1,k_2) = \inf_{\{c \in \mathcal S_{k_1,k_2}\}} l(c),$
where 
$
\mathcal S_{k_1,k_2} = \{c : c(0)=k_1, c(T) =k_2, c \text{\ \ admissible} \}
$,  \cite{ABB,calin,J}.
However, since minimizing of the energy of a curve implies minimizing of its length, we will rather minimize energy of curves.
Altogether, we would like to study the following optimal control problem 
\begin{align} \label{control_orig}
\dot c (t)=u_1X_1+u_2X_2 =u_1 \left(
\begin{array}{c}
0\\
0\\
1\\
\end{array}
\right) 
+u_2 \left(
\begin{array}{c}
\cos\theta\\
\sin\theta\\
0\\
\end{array}
\right)
\end{align}
for $c$ in $K$ and the control $u=(u_1,u_2)$ with the boundary condition $c(0)=k_1,$ $c(T)=k_2$,
where we minimize ${1 \over 2}\int_0^T (u_1^2+u_2^2) dt$.

\subsection{Left-invariant Hamiltonians}
We use the left invariance of the control system \eqref{control_orig} and we view  it as a control problem on a Lie group, \cite{ABB}. The tangent bundle $TK$ and the cotangent bundle $T^*K$ are trivializable via the left-invariant fields $X_1, X_2, X_3:= X_{12}$. This means that we can identify each left-invariant object on $K$ with its value at the origin $o \in K$. In particular, we identify $T_oK \simeq \fk$, where the generators $e_i,$ $i=1,2,3$ of $\fk$ correspond to $X_i(o),$ $i=1,2,3$. For each vector field $v$ on $K$ and $g \in K$, the vector  $v(g) \in T_gK$ corresponds to $(g,\mu) \in K \times \fk$, where $\mu=T\ell_{g^{-1}}.v \in T_oK \simeq \fk$ and  $\ell_g$ denotes the left multiplication given by the group structure in $K$. 
Analogously, $T_o^*K \simeq \fk^*$ for the dual basis of $X_i$, $i=1,2,3$ and the generators $e_i^*$ of $\fk^*$ dual to $e_i$, $i=1,2,3$. For each one-form $f$ on $K$ and $g \in K$, the covector $f(g) \in T^*_gK$ corresponds to $(g,\xi) \in K \times \fk^*$, where $\xi=T^*\ell_{g^{-1}}.f \in T^*_oK \simeq \fk^*$. Then $f(g)=\sum_i h_i T^*\ell_g.e_i^*$ for suitable functions $h_i$ called \emph{vertical coordinates} and we can write $\xi=\sum_i h_i e_i^*$, where functions $h_i$ are determined by the value at $o$ by left action. 
\begin{lem*} \label{1.2}
It holds $h_i(f(g))=\langle X_i(g), f \rangle$ and functions $h_i:T^*K \to \R$ are independent of the choice $g \in K$. Their Poisson bracket is given as $\{ h_i, h_j \}= \langle [X_i,X_j], f \rangle=\langle [e_i, e_j], \xi \rangle$, where $\langle\ , \ \rangle$ denotes the evaluation.
\end{lem*}
\begin{proof}
We have $\langle X_i(g), f \rangle=\langle e_i,T^*\ell_{g^{-1}}.f \rangle = \langle e_i, \xi \rangle=\langle e_i, \sum_j h_je_j^* \rangle = h_i$. The rest follows.
\end{proof}
Consider arbitrary Hamiltonian $H: T^*K \to \R$. Identification $T^*K \simeq K \times \fk^*$ allows to find a function $\HH:K \times \fk^* \to \R$ such that 
$\HH(g,\xi)=H((T^*\ell_g.\xi)(g))$. Then the Hamiltonian $H$ is left-invariant if the function $\HH$ does not depend on $g \in K$ and one can view it as a function $\HH : \fk^* \to \R $. For $f=\sum_i h_i T^*\ell_g e_i^*$ we then have $H(f(g))=\HH(\sum_i h_i e_i^*)$.

\subsection{Pontryagin's maximum principle}

We say that the pair $(\hat u(t),\hat g(t))$ of a control and a trajectory is an  \emph{optimal pair} if $\hat g(t)$ is a length minimizer and satisfies $\dot g=u_1X_1(g)+u_2X_2(g)$, $g \in K$ with the control function $u=\hat u(t)$.
\begin{def*}
The \emph{Hamiltonian of the maximum principle} is a family of smooth functions parametrized by controls $u=(u_1,u_2) \in \R^2$ and a real number $\nu \leq 0$ given by 
$$
H(\nu,f)=\langle u_1X_1+ u_2 X_2, f \rangle+{\nu \over 2}(u_1^2+u_2^2)=u_1h_1+u_2h_2 +{\nu \over 2}(u_1^2+u_2^2).
$$
\end{def*}
If $\nu=0$, we speak about \emph{abnormal Hamiltonian}. Otherwise, we speak about \emph{normal Hamiltonian}. Then we can normalize $H$ by the choice $\nu=-1$. 

\begin{thm*}[PMP, \cite{AS,ABB}]  If a pair $(\hat u(t),\hat g(t))$ is optimal in a minimization problem, then there exists a Lipschitzian curve $f(t) \in T^*_{\hat g(t)}M$ and a number  $\nu \leq 0$ such that the following hold:
\begin{itemize}
\item $(f(t),\nu) \neq 0$,
\item $\dot f(t)=\vec H_{\hat u(t)}(f(t))$, where $\vec H$ is the Hamiltonian vector field of $H$, 
\item $H_{\hat u(t)}= \textrm{max}_{u} H_u(f(t),\nu).$
\end{itemize}
\end{thm*}
The curve $f(t)$ is usually called \emph{extremal}. The extremal is called \emph{normal} (resp. \emph{abnormal}) if it corresponds to normal (resp. abnormal) Hamiltonian of the maximum principle. It follows from \cite{J} that the projection of abnormal to $K$ always is also  projection of a normal, so there are no strict abnormals for $1$--step filtrations.  In the next, we focus only on normal extremals. Let us note that abnormal extremals are often studied on longer filtrations, \cite{bm, s2358}.

According to the third condition, the extreme of the normal Hamiltonian of PMP is achieved when ${\partial H \over \partial u_i}=h_i-u_i=0$ for $i=1,2$ and this implies that $u_i=h_i$, $i=1,2$ for the controls.  In this case, the Hamiltonian of the maximum principle is  
\begin{align} \label{ham}
H={1 \over 2}(h_1^2+h_2^2)
\end{align}
and this is left-invariant.

The Hamiltonian system of PMP for a sub--Riemannian structure $(K,\D,k)$ is given by 
\begin{align}
&\dot h_i=\{ H,h_i \}, \label{ver_obec} \\
& \dot g= h_1X_1(g)+h_2X_2(g) \label{hor_obec}
\end{align}
and it follows from  the left-invariance that the vertical system \eqref{ver_obec} takes the form 
\begin{align} \label{ad}
\dot h_i=
\langle (\ad\  d\HH)e_i, \xi \rangle = \langle e_i, (\ad\  d\HH)^* \xi \rangle,
\end{align} 
where $\HH:\fk^* \to \R$ corresponds to the Hamiltonian of the maximum principle $H$ and $\xi \in T^*_oK$, \cite{AS}.

\begin{prop*} \label{extr-orig}
Normal extremals of the rolling disc problem $(\ref{ver_obec},\ref{hor_obec})$ are solutions of the system 
\begin{align} \label{ver}
&\dot h_1= -h_3h_2, \ \ \ \ 
\dot h_2= h_3h_1, \ \ \ \ 
\dot h_3= -h_1h_2,\\ \label{hor}
&\dot x=h_2\cos\theta,\ \ \ \ 
\dot y=h_2\sin\theta,\ \ \ \ 
\dot \theta = h_1,
\end{align}
where \eqref{ver} is the vertical system and \eqref{hor} is the horizontal system.
\end{prop*}
\begin{proof}The left invariant Hamiltonian $H$ from equation \eqref{ham} satisfies $dH=h_1 dh_1+h_2 dh_2$ and thus $d\HH=h_1 e_1+h_2 e_2$ in formula \eqref{ad}. Direct computation gives that the adjoint action $\ad(h_1 e_1+h_2 e_2)$ viewed as a linear endomorphism is represented in the basis $e_i$, $i=1,2,3$ by the matrix 
$\left( \begin{smallmatrix} 
0 & 0 & 0\\
0 & 0 & -h_1 \\
-h_2 & h_1 & 0
\end{smallmatrix} \right)$. Then we read of directly the system \eqref{ver} from the action of this matrix. The horizontal system \eqref{hor} follows from the form of the generators $X_1$ and $X_2$ of $\D$.
\end{proof}

In the Listing \ref{cod_adj} we compute the adjoint action of the controllability algebra on three--dimensional vector space $VS$.

\tiny 
\begin{lstlisting}[language=Mathematica ,caption={Adjoint represenation}, label={cod_adj}]
DGsetup([v1, v2, v3], VS);
DGsetup(Alg);
LieAlgebras:-Adjoint(Alg, representationspace = VS);
LieAlgebras:-Adjoint(e1*h1 + e2*h2);
\end{lstlisting}
\normalsize
The system (\ref{ver},\ref{hor}) can be solved and the solution can be found in \cite{samo}. However, the solution is formulated in the language of Jacobi elliptic functions that are hard to use in implementations. Thus we swap in our next considerations to nilpotent approximation which is simpler model but still describes the system appropriately.

\section{Homogeneous nilpotent approximation of disc} \label{sec_nilp}

The homogeneous nilpotent approximation is an approximation of the control distribution such that it is spanned by a nilpotent basis but the geometric structure on the configuration space is preserved. By taking Lie brackets it defines a nonholonomic tangent space that can be regarded as the ``principal part''  of the structure defined on the configuration space by the distribution in a neighbourhood of a point, for details see \cite{ABB}.  An algebraic construction of the homogeneous nilpotent approximation was developed by various researchers, s.e.g. \cite{as87,h86}. Roughly said, it is obtained by taking Taylor polynomials of suitable degrees of coefficients of the control vector fields expressed in privileged coordinates. For a definition of privileged coordinates see \cite{ABB, bel} or the references above. 

We have shown in Section \ref{control} that the rolling disc distribution $\D$ is equiregular with the growth vector $(2,3)$ at every point. As we shall see the nonholonomic tangent space is unique and given by the Heisenberg group in this case.

\subsection{Construction of nilpotent approximation}

The distribution of the rolling disc is defined  in Section \ref{definice}
 by means of vector fields \eqref{fields} that reads
\begin{align*}
X_1 &=\partial_{\theta}, \\
X_2 &=\cos\theta\partial_x+\sin\theta \partial_y.
\end{align*}
In this subsection, we consider the coordinates in the order $(\theta,x,y)\in \mathbb S^1\times\mathbb{R}^2$.
The reason is that this coordinate system is already privileged with weights $(1,1,2)$ and the nilpotent approximation is obtained simply by the usual engineering approximation $\cos\theta\approx 1$, $\sin\theta\approx \theta.$ 
\begin{lem*} \label{NA}
The nilpotent approximation at the origin $(0,0,0)$ of the rolling disc is given by the distribution generated by
\begin{align} \label{nilp}
\begin{split} 
 n_1 &=\partial_{\theta}, \\ 
 n_2 &=\partial_x + \theta \partial_y 
\end{split}
\end{align}
 with the only non--trivial bracket $[n_1,n_2]=n_3:=\partial_y$. Thus the symbol is the Lie algebra  
 $$\mathfrak m:= ( \langle n_1,n_2,n_3 \rangle, [n_1,n_2]=n_3 ),$$
 which is the graded nilpotent Lie algebra 
$\mathfrak m = \mathfrak g^{-2} \oplus  \mathfrak g^{-1}$, 
where $\mathfrak g^{-2} = \langle n_3 \rangle$ and   $\mathfrak g^{-1} = \langle n_1, n_2 \rangle$.
\end{lem*}
\begin{proof}
The fact that coordinates $(\theta,x,y)$ are privileged at $(0,0,0)$ can be seen either from the Maclaurin expansion of coefficients of vector fields \eqref{fields} or checked directly by computing nonholonomic derivatives of the coordinates at the origin. Indeed, in $(0,0,0)$ we have $X_1(\theta)=1$, $X_2(x)=1$ while $X_1(y)=X_2(y)=0$ and $[X_1,X_2](y)=1.$ Thus the coordinates $(\theta,x,y)$ are privileged with weights $(1,1,2)$ and the nilpotent approximation is given by Taylor polynomials of corresponding weighted order. Namely, coefficients of  a vector field approximating any of \eqref{fields} must be of order 0, 0 and 1,  respectively. In other words, the first two coefficients (the direction of the  distribution) must be constant while the third one must be linear (in $x$ or $\theta$). Hence we get $n_1=(1,0,0)$, $n_2=(0,1,\theta).$ The rest follows.
\end{proof}

In the Listing \ref{cod_nilp} we define nilpotent Lie algebra and discuss its properties.
\tiny 
\begin{lstlisting}[language=Mathematica ,caption={Nilpotent aproximation},label={cod_nilp}]
DGsetup([x, y, theta], N);
n1 := evalDG(D_theta);
n2 := evalDG(D_y*theta + D_x);
n3 := LieBracket(n1, n2);
NAlg := LieAlgebraData([n1, n2, n3]);
DGsetup(NAlg);
Query("Solvable");
Query("Nilpotent");
\end{lstlisting}
\normalsize

Our coordinates on the configuration space $M:=\mathbb S^1\times\mathbb{R}^2$ are actually so called \emph{canonical coordinates of the 2nd kind} for the nilpotent Lie algebra $\mathfrak m,$ and are given by the local isomorphism 
\begin{align}\label{2nd}
\Psi_2 : \mathbb R^3 \to M, \:\:\: \Psi_2(\theta,x,y)
=e^{\theta n_1}\circ e^{x n_2}\circ e^{y n_3} (q),
\end{align} 
where $e$ denotes the exponential mapping and $q$ is the point where the map is centered. Indeed, the nilpotent system \eqref{nilp} can be recovered from \eqref{2nd} by differentiation and solving 
\begin{align*}
\partial_\theta&=\Psi^{-1}_{2*}n_1,\\
\partial_x&=\Psi^{-1}_{2*}n_2-\theta\Psi^{-1}_{2*}n_3,\\
\partial_y&=\Psi^{-1}_{2*}n_3.
\end{align*}
Hence the nilpotent system of vectors \eqref{nilp} that we obtained from the nilpotent approximation procedure is the normal form for the generating family of the nonholonomic tangent space. These vector fields are also often called generators of the polarized Heisenberg group, \cite{ABB}. 

\begin{rem} 
Note that the generating family of the nonholonomic tangent space can have various forms. For example, the coordinates defined as
\begin{align}
\Psi_1 : \mathbb R^3 \to M, \:\:\: \Psi_1(\theta,x,y)
=e^{yn_3+ xn_2+ \theta n_1} (q),
\end{align}
are so called \emph{canonical coordinates of the 1st kind}. The generating family  reads as
\begin{align}
n_1= \partial_{\theta} - \frac{1}{2} x  \partial_{y} , \:\:\:\: n_2=\partial_{y} + \frac{1}{2} \theta \partial_{y}.
\end{align}
These vector fields are also often called the standard generators of the Heisenberg group.
The diffeomorphism  between these two coordinates is given by the change of variable $y\mapsto y+\frac12 \theta x.$
\end{rem}

The inverse of the map \eqref{2nd} also defines a group structure on $\mathbb R^3$ by taking compositions on $M$. 
This group is called \emph{nonsymmetric 3--dimensional Heisenberg group} or \emph{Heisenberg group}. The  group multiplication looks as follows.
\begin{lem*} \label{grupa}
The space  $\mathbb R^3$  together with 
the non-commutative group law 
\begin{align} \label{group}
(\theta, x,y) \circ (\tilde \theta, \tilde x,\tilde y)
=  (\theta + \tilde \theta, x + \tilde x, y +\tilde y +  \theta \tilde x ) 
\end{align} 
forms connected, simply connected, non--commutative Lie group $N$. 
The vector fields $n_1,n_2$ and $n_3$ are left invariant with respect to the Lie group law \eqref{group}.
\end{lem*}
\begin{proof} Taking derivatives of \eqref{group} at zero with respect to $\tilde\theta,$ $\tilde x$ and $\tilde y$ we get $(1,0,0)$, $(0,1,\theta)$ and $ (0,0,1)$ respectively. For details see \cite{calin}. \end{proof}

Having an explicit formula for the group multiplication we easily find the corresponding right--invariant vector fields.
\begin{prop*} \label{pravo-inv}
The right--invariant vector fields with respect to the Lie group law \eqref{group} on $\mathbb R^3$ read
\begin{align}
\begin{split}
R_1 &= \partial_{\theta}   +x \partial_{y},\\
R_2 &= \partial_{x}, \\
R_{3} & = -[R_1,R_2] = \partial_y.
\end{split}
\end{align}
\end{prop*}
\begin{proof}
Taking derivatives of \eqref{group} at zero with respect to $\theta,$ $x$ and $y$ we get $(1,0,\tilde x)$, $(0,1,0)$ and $ (0,0,1)$ respectively.
\end{proof}

In the Listing \ref{cod_group} we define Heisenberg Lie group and discuss  invariant vectors and forms. 

\tiny 
\begin{lstlisting}[language=Mathematica ,caption={Heisenberg Lie group}, label={cod_group}]
DGsetup([x1, x2, x3], L);
T := Transformation(L, L, [x1 = x1 + y1, x2 = x2 + y2, x3 = x1*y2 + x3 + y3]);
DGsetup([y1, y2, y3], G);
LG := LieGroup(T, G);
InvariantVectorsAndForms(LG);
\end{lstlisting}
\normalsize

\subsection{Formulation of the nilpotent problem and PMP} \label{pmp}
Let us now discuss the nilpotent control problem that approximates the problem from Section \ref{pmp-orig}. 
The vector fields $n_1, n_2$ generate the left-invariant horizontal distribution $\N$ on the Lie group $N$ from Lemma \ref{grupa} 
which then carries a canonical filtration with the growth $(2,3)$. 
Let us return back to the classical setting from Section \ref{definice} and in particular, consider the coordinates in the original order $(x,y,\theta)$. 
Then we define the sub--Riemannian metric $r$ in $\N=\langle n_1,n_2\rangle$ such that the fields $n_1$ and $n_2$ are orthonormal  with respect to $r$. Thus 
\begin{align} \label{metr}
r= \text{d}x^2 + \text{d}\theta^2.  
\end{align}

\begin{lem*}
The sub--Riemannian structure $(N,\N,r)$ is left-invariant with respect to the group structure \eqref{group}.
\end{lem*}

In the Listing \ref{cod_met} we define  sub--Riemannian control metric $r$. 

\tiny 
\begin{lstlisting}[language=Mathematica ,caption={Sub--Riemannian metrics},label={cod_met}]
ChangeFrame(N);
ndb := DualBasis([n1, n2, n3]);
r := evalDG((ndb[1] &t ndb[1]) + (ndb[2] &t ndb[2]));
\end{lstlisting}
\normalsize

Consider two points $k_1, k_2 \in N$. Among all admissible curves $c(t)$, we would like to find length minimizers with respect to $r$. 
Thus we will study the following optimal control problem 
\begin{align} \label{control_nil}
\dot c(t)=u_1n_1+u_2n_2 =u_1 \left(
\begin{array}{c}
0\\
0\\
1\\
\end{array}
\right) 
+u_2 \left(
\begin{array}{c}
1\\
\theta\\
0\\
\end{array}
\right)
\end{align}
for   $c$ in $N$ and the control $u=(u_1,u_2)$ with the boundary condition $c(0)=k_1,$ $c(T)=k_2$,
where we minimize ${1 \over 2}\int_0^T (u_1^2+u_2^2) dt$. 

We use Hamiltonian viewpoint and Pontryagin's maximum principle to stu\text{d}y the problem \eqref{control_nil}. 
Since the control system is left-invariant, we use methods introduced in Section \ref{pmp-orig} to find the system giving local extremals.
We consider the cotangent bundle $T^*N \to N$ trivialized by duals of $n_1, n_2, n_3$ and use corresponding vertical coordinates $h_1, h_2, h_3$ and horizontal coordinates $x,y, \theta$. 

\begin{prop*} \label{sys}
Normal extremals of the approximation are solutions of the system 
\begin{align} \label{ver_nil}
&\dot h_1= -h_3h_2, \ \ \ \ 
\dot h_2= h_3h_1, \ \ \ \ 
\dot h_3=0,\\ \label{hor_nil}
&\dot x=h_2,\ \ \ \ \ \ \ \ \ \ 
\dot y=\theta h_2,\ \ \ \ \ \ \ \ \ 
\dot \theta = h_1,
\end{align}
where \eqref{ver_nil} is the vertical system and \eqref{hor_nil} is the horizontal system.
\end{prop*}
\begin{proof}
Analogously to the proof of Proposition \ref{extr-orig}, the normal Hamiltonian of the maximum principle is  
$H={1 \over 2}(h_1^2+h_2^2)$, and then $dH=h_1 dh_1+h_2 dh_2$ and $d\HH=h_1 e_1+h_2 e_2$, where $e_i \in \fm$ correspond to fields $n_i$. Direct computation gives the adjoint action $\ad(h_1 e_1+h_2 e_2)$ by the matrix 
$\left( \begin{smallmatrix} 
0 & 0 & 0\\
0 & 0 & 0 \\
-h_2 & h_1 & 0
\end{smallmatrix} \right)$. Then we read of the explicit form of the system \eqref{ver} from the action of this matrix. The explicit form of the horizontal system \eqref{hor} follows from the form of the generators $n_1$, $n_2$ of $\N$.
\end{proof}

In the Listing  \ref{cod_Adj} we compute the adjoint action of Heisenberg Lie algebra 
on three--dimensional vector space $VS$. 
\tiny 
\begin{lstlisting}[language=Mathematica ,caption={Adjoint representation},label={cod_Adj}]
DGsetup(NAlg);
LieAlgebras:-Adjoint(NAlg, representationspace = VS);
LieAlgebras:-Adjoint(e1*h1 + e2*h2);
ADJ := LieAlgebras:-Adjoint(e1*h1(t) + e2*h2(t));
\end{lstlisting}
\normalsize

\subsection{Solutions of the system}
The system from Proposition \ref{sys} can be solved explicitly. Let us start with the vertical system \eqref{ver_nil}, which is independent of the horizontal part.
\begin{prop*} \label{solution} In the generic case $h_3 \neq 0$, 
the system \eqref{ver_nil} has the solution  
\begin{align} \label{vert_sol}
\begin{split} 
h_1&=C_2\sin(C_1t)+C_3\cos(C_1t), \\
h_2&=C_3\sin(C_1t)-C_2\cos(C_1t), \\
h_3&=C_1
\end{split} 
\end{align}
for constants $C_1,C_2,C_3$. In the case $h_3=0$ we get $h_1=C_1$ and $h_2=C_2$ for constants $C_1, C_2$.
\end{prop*}
\begin{proof}
The equation $\dot h_3=0$ implies $h_3=C_1$ for some constant $C_1$. If $C_1=0$ then $h_2=C_2$ and $h_1=C_3$ for constants $C_2, C_3$.
If $C_1 \neq 0$, then we get a system 
$$
\left( \begin{smallmatrix} 
\dot{h}_1\\
\dot{h}_2
\end{smallmatrix} \right)=
\left( \begin{smallmatrix} 
0 &-C_1 \\
C_1 & 0
\end{smallmatrix} \right)
\left( \begin{smallmatrix} 
h_1\\
h_2
\end{smallmatrix} \right)
$$
with constant coefficients and we get the solution \eqref{vert_sol} by the discussion of eigenvalues and eigenvectors of corresponding matrix.
\end{proof}

The horizontal system then can be solved by direct integration. Moreover, it is sufficient to solve the system with the initial condition $x(0)=y(0)=\theta(0)=0$, because we get solutions with different initial conditions using the group structure of $N$. We discuss this in detail in Section \ref{nilp_sym}. Altogether, we get the following statement.

\begin{prop*} \label{prop_hor}
In the case $h_3 \neq 0$, the horizontal system \eqref{hor_nil} has solutions satisfying 
$x(0)=y(0)=\theta(0)=0$   
\begin{align} \label{sol_nil}
\begin{split}
x &= {1 \over C_1} \big (C_3-C_2\sin(C_1t)-C_3\cos(C_1t) \big ),\\
y &= 
{1 \over 4C_1^2} \big (
2C_1(C_2^2+C_3^2)t-4C_2C_3\cos(C_1t)+2C_2C_3\cos(2C_1t)\\
&-4C_2^2\sin(C_1t)+(C_2^2-C_3^2)\sin(2C_1t)+2C_2C_3 \big ),\\
\theta &= {1 \over C_1} \big (C_2 -C_2\cos(C_1t)+C_3\sin(C_1t) \big )
\end{split}
\end{align}
for constants $C_1$, $C_2$, $C_3$ from Proposition \ref{solution}.
In the degenerate case $h_3=0$ we get $x = C_2t$, $y={1 \over 2} C_2 C_1 t^2$, $\theta = C_1t$ for $C_1$, $C_2$ from Proposition \ref{solution}.
\end{prop*}

In the Listing \ref{cod_vert} we find the system of the approximation and its solution.
\tiny 
\begin{lstlisting}[language=Mathematica ,caption ={Approximated system},   label ={cod_vert}]
SV := map(diff, Matrix([h1(t), h2(t), h3(t)]), t) -
MatrixMatrixMultiply(Matrix([h1(t), h2(t), h3(t)]), ADJ)
sys_vert := {SV[1, 1], SV[1, 2], SV[1, 3]};
sol_vert := pdsolve(sys_vert);
pdetest(sol_vert[1], sys_vert);
pdetest(sol_vert[2], sys_vert);
GC := subs(x = x(t), y = y(t), theta = theta(t), 
GetComponents(evalDG(h1(t)*n1 + h2(t)*n2), DGinfo("FrameBaseVectors")));
# The order of variables can differ in diferent versions of Maple. Maple 2019 
# follows alphabetical order and which is reflected in sys.  
sys := {op(sys_vert), diff(theta(t), t) = GC[3], 
diff(x(t), t) = GC[1], diff(y(t), t) = GC[2]};
sol := pdsolve(sys);
pdetest(sol[1], sys);
pdetest(sol[2], sys);
\end{lstlisting}
\normalsize

In the Listing \ref{cod_poc} we find the solution of the system satisfying the initial condition from Proposition \ref{prop_hor}. We apply this only on the second solution set (which is non--degenerate one in the case of Maple 2019). 

\tiny 
\begin{lstlisting}[language=Mathematica ,caption ={Initial condition},   label ={cod_poc}]
sol2 := [sol[2][4], sol[2][5], sol[2][6]];
poc2 := eval(subs(t = 0, sol2));
psol2 := solve({rhs(poc2[1]), rhs(poc2[2]), rhs(poc2[3])});
sol20 := simplify(subs(psol2, sol2));
eval(subs(t = 0, sol20));
\end{lstlisting}
\normalsize

Finally, different extremals can project to the same unparametrized local minimizers in $N$. Let us consider only curves with fixed parametrization.  
Local minimizers  with the same parametrizations correspond to level-sets of the normal Hamiltonian, \cite{ABB}, and local minimizers parametrized by arc-length are exactly projections of extremals satisfying the condition $h_1^2+h_2^2=1$. This condition takes the form $C_1^2+C_2^2=1$ in the case $h_3=0$ and $C_2^2+C_3^2=1$ in the  case  $h_3 \neq 0$. In the next, we will focus on the generic solutions \eqref{sol_nil}.

\subsection{Examples and comparison to the original system}
Let us demonstrate on examples that solutions of the nilpotent system approximate the solutions of the original system in the neighbourhood of the origin.
In the following examples, we will solve the original system numerically in Maple and we will compare the solutions with solutions of the nilpotent system with the same initial condition.

\begin{example} \label{ex1}
Let us consider the initial condition $x(0)=y(0)=\theta(0)=0$, $h_1(0)={1 \over 2}$, $h_2(0)={\sqrt{3} \over 2}$, $h_3(0)=2$.
The nilpotent system has the solution    
$h_1 ={1 \over 2}\cos(2t)-{\sqrt{3} \over 2}\sin(2t)$,
$ h_2 = {1\over 2}\sin(2t)+{\sqrt{3}\over 2}\cos(2t)$,
 $h_3 = 2$, 
$ x ={1 \over 4}+{\sqrt{3} \over 4}\sin(2t)-{1\over 4}\cos(2t)$, 
$y = -{\sqrt{3}\over 16}\cos^2(2t)+{\sqrt{3}\over 16}\cos(2t)+{1\over 16}\cos(2t)\sin(2t)-{3\over 16}\sin(2t)+{t\over 4}$, 
$
\theta = {\sqrt{3}\over 4}\cos(2t)+{1\over 4}\sin(2t)-{\sqrt{3}\over 4}$.

For numeric solutions of the original system with the same initial condition Maple generically uses Runge-Kutta Fehlberg method and we compute the solution on the interval $[0,\pi]$ with the step ${\pi \over 100}$. We display the local minimizers for both systems in Figure \ref{disk1a}, where the numeric solution of the original system is blue (dot line), and the solution of the nilpotent system is red (solid line). 
For better mechanical illustration we present the angular and plane movements in Figures   \ref{disk1b} and \ref{disk1c}.
\end{example}

\begin{figure}[h] 
	\subfigure[Local mimimizer]{\includegraphics[width=0.3 \textwidth]{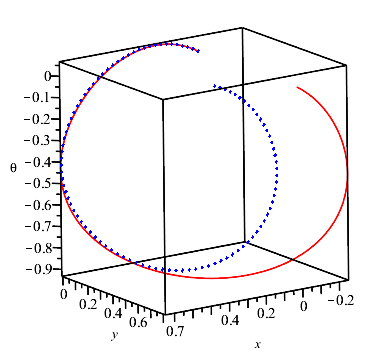} 
	\label{disk1a}}
	\subfigure[The parameter $\theta$]{\includegraphics[width=0.3
	\textwidth]{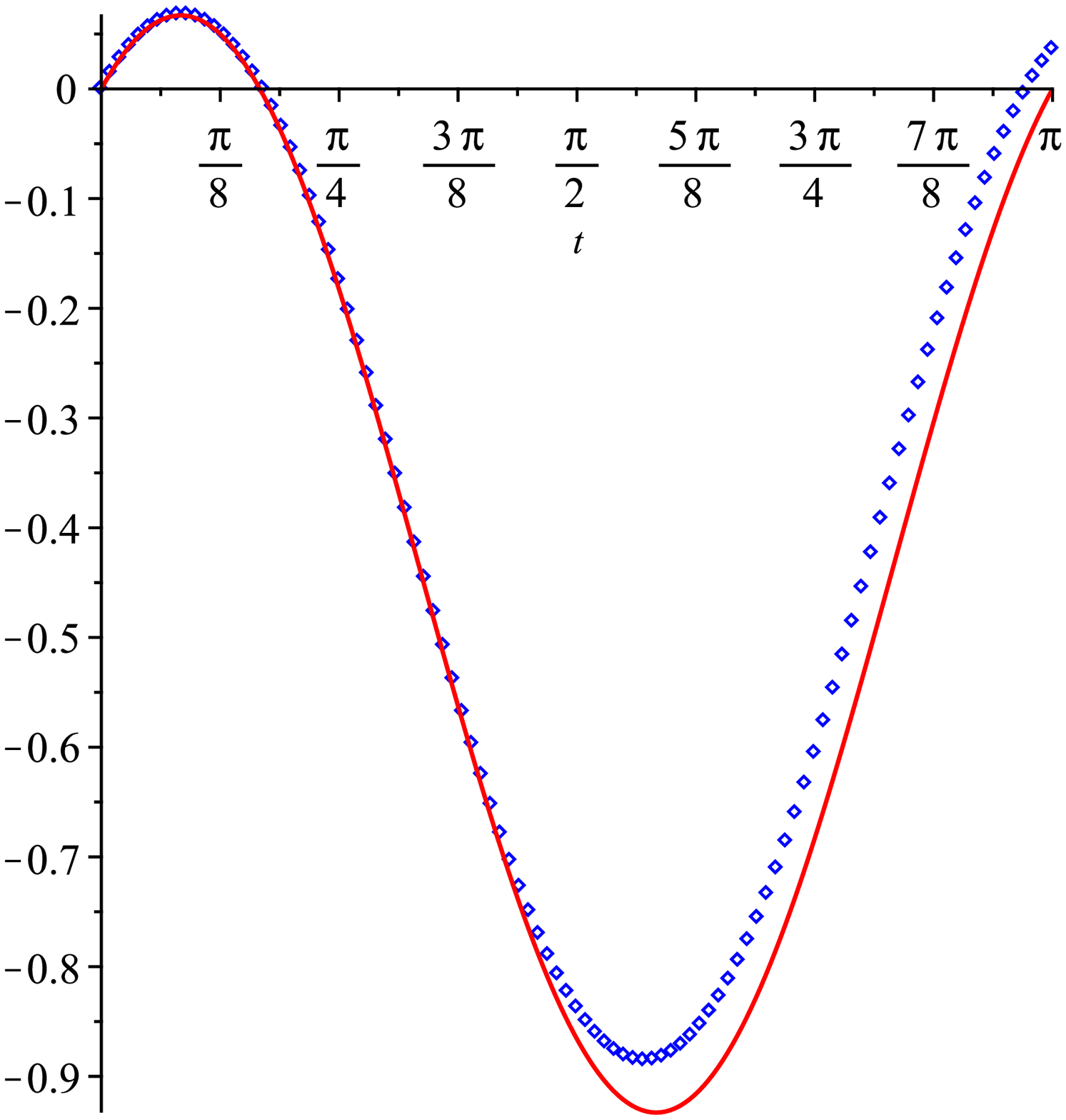} \label{disk1b}}
	\subfigure[Trajectory of the contact pont in the plane $(x,y)$]{\includegraphics[width=0.3
	\textwidth]{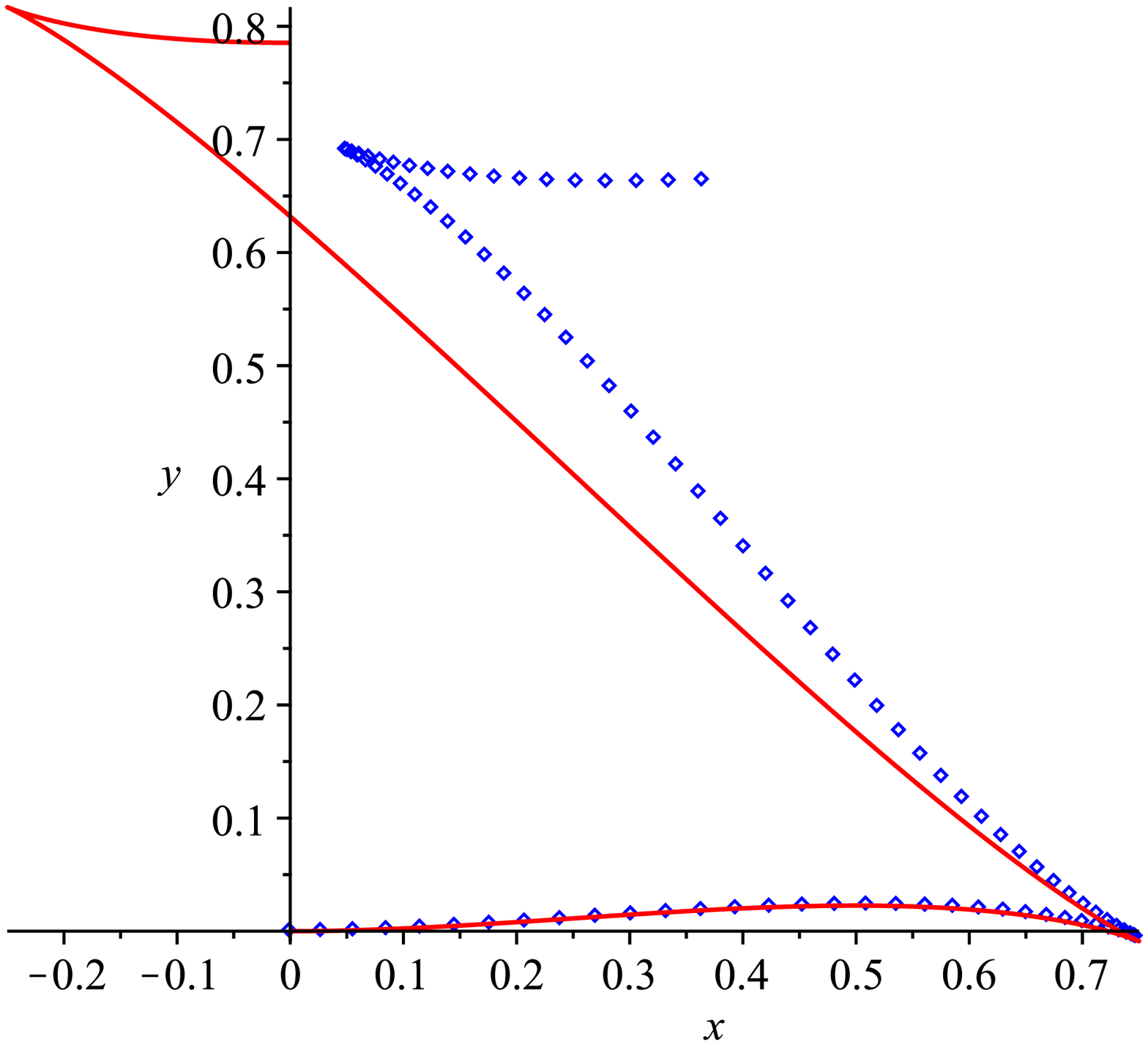} \label{disk1c}}
	\caption{Comparing of the analytic solution of approximation and the numeric solution of the original system: Example 1}
	\label{disk1}
\end{figure}

\begin{example}
Let us now consider the initial condition $x(0)=y(0)=\theta(0)=0$, $h_1(0)={1 \over 2}$, $h_2(0)={\sqrt{3} \over 2}$, $h_3(0)=20$.
The nilpotent system has the solution    
$h_1 ={1 \over 2}\cos(20t)-{\sqrt{3} \over 2}\sin(20t)$,
$ h_2 = {1\over 2}\sin(20t)+{\sqrt{3}\over 2}\cos(20t)$,
 $h_3 = 20$, 
$ x ={1 \over 40}+{\sqrt{3} \over 40}\sin(20t)-{1\over 40}\cos(20t)$, 
$y = -{\sqrt{3}\over 1600}\cos^2(20t)+{\sqrt{3}\over 1600}\cos(20t)+{1\over 1600}\cos(20t)\sin(20t)-{3\over 1600}\sin(20t)+{t\over 40}$, 
$
\theta = {\sqrt{3}\over 40}\cos(20t)+{1\over 40}\sin(20t)-{\sqrt{3}\over 40}$.

For numeric solutions of the original system with the same initial condition, we use Maple with the step ${\pi \over 1000}$ and we compute it on the interval $[0,\frac{\pi}{10}]$.  We display the local minimizers for both systems in Figure \ref{disk2a} and  we present the angular  and plane  movements in Figures   \ref{disk2b} and \ref{disk2c}.

\end{example}

\begin{figure}[h] 
	\subfigure[Local mimimizer]{\includegraphics[width=0.26\textwidth]{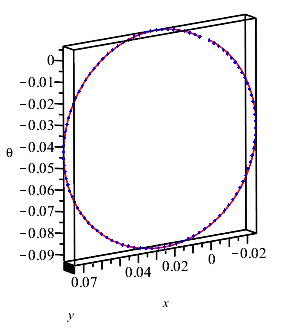} \label{disk2a}}
	\subfigure[The parameter $\theta$]{\includegraphics[width=0.3\textwidth]{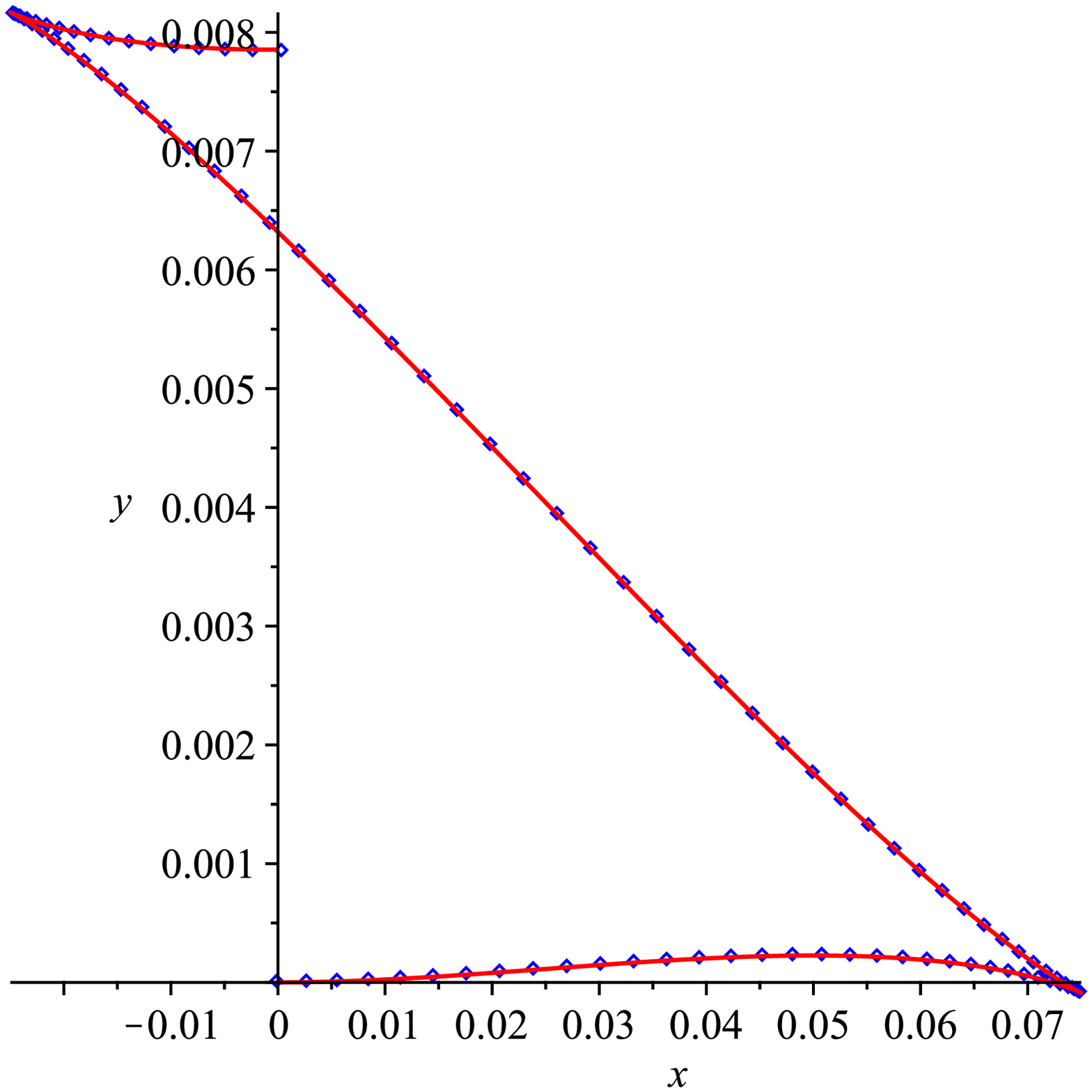}\label{disk2b}}
	\subfigure[Trajectory of the contact pont in the plane $(x,y)$]{\includegraphics[width=0.3
	\textwidth]{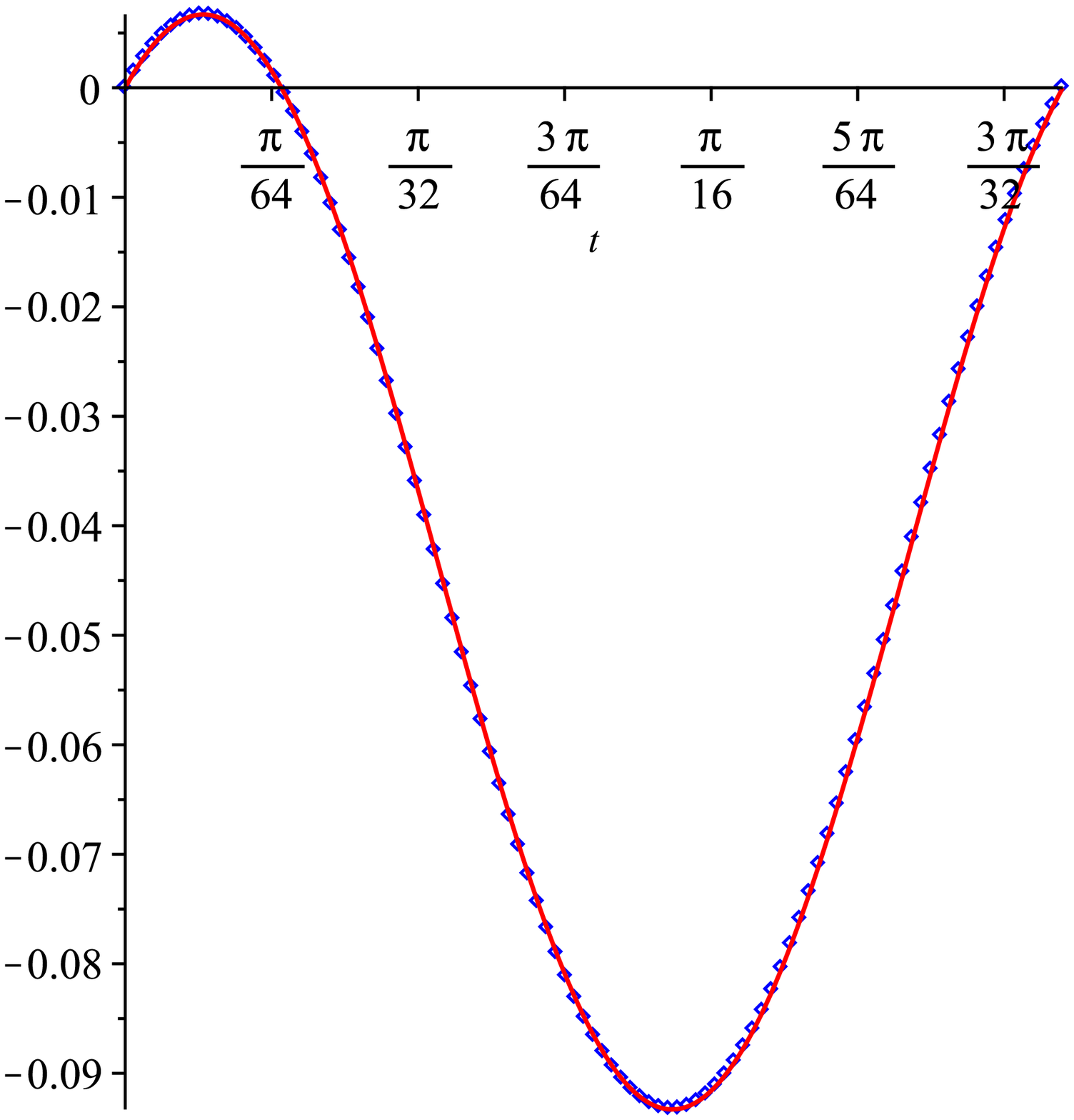} \label{disk2c}}
	\caption{Comparing of the analytic solution of approximation and the numeric solution of the original system: Example 2}
	\label{disk2}
\end{figure}

Let us note that our choice is such that in both cases, we display one period of the graphs and we have evenly distributed $100$ points in the interval.

\section{ Symmetries}  \label{nilp_sym}

Let us discuss here infinitesimal symmetries of the control structures. By an \emph{infinitesimal symmetry} we mean a vector field such that its flow is a symmetry of the geometric structure at all times.
\subsection{Symmetries of the control system} We can find explicitly all infinitesimal  symmetries of the sub--Riemmanian structure $(N,\N,r)$. Indeed, we are interested in vector fields $v$ such that $\mathcal{L}_v(\N) \subset \N$ and $\mathcal{L}_v(r)=0$. This gives a system of pde's that can be solved explicitly in the case of the left-invariant nilpotent structure.

\begin{prop*} \label{sym_met}
Infinitesimal symmetries of the left-invariant sub--Riemmanian structure $(N,\N,r)$ form a Lie algebra generated (over $\R$) by vector fields 
\begin{align} \label{subalg}
\begin{split}
&t_0:=\theta\del_x+{\theta^2-x^2 \over 2}\del_y-x\del_\theta,\\
&t_1 :=\del_x,\\
&t_2 :=x\del_y+\del_\theta,\\
&t_3 := \del_y.           
\end{split}
\end{align}
In particular, $t_0$ generates the isotropic subalgebra at $(0,0,0)$. Fields $t_i$, $i=1,2,3$ are translations that reflect the Heisenberg structure and coincide with the right-invariant fields, see Proposition \ref{pravo-inv}.
\end{prop*}
\begin{proof}
We write the corresponding pde's using the contact form $\phi=\text{d}y-\theta \text{d}x$ of $\N$ as follows. For arbitrary vector field $v=f_1(x,y,\theta)n_1+f_2(x,y,\theta)n_2+f_3(x,y,\theta)n_3$ we compute the Lie derivative of $n_1,n_2$ with respect to $v$ and evaluate it on $\phi$. This gives the system 
\begin{align} \label{sys_dist}
\begin{split}
f_2+{\del f_3 \over \del \theta}=0, \ \ \ \ \ \
f_1-{ \theta{\del f_3 \over \del y}} - {\del f_3 \over \del x}  =0
\end{split}
\end{align}
on infinitesimal symmetries preserving the distribution $\N$. Moreover, Lie derivative of the metric $r$ has the form
\begin{align} \label{L_met}
2{\del f_2 \over \del x} \text{d}x^2 +2{\del f_2 \over \del y} \text{d}x \text{d}y+ 2\left( {\del f_1 \over \del x}+{\del f_2 \over \del \theta} \right)\text{d}x \text{d}\theta + 2{\del f_1 \over \del y} \text{d}y \text{d}\theta +2{\del f_1 \over \del \theta} \text{d}\theta^2
\end{align}
and this tensor should vanish. The vanishing of coefficients of \eqref{L_met} implies that $f_1$ is a function of $x$ and $f_2$ is a function of $\theta$. Moreover, since the derivative of $f_1$ with respect to $x$ (which is a function of $x$) equals to minus the derivative of $f_2$ with respect to $\theta$ (which is a function of $\theta$) both functions $f_1,f_2$ must be linear with the same leading coefficient up to sign. Thus $f_1=-C_1x+C_3$ and $f_2=C_1\theta+C_2$. Then first equation of \eqref{sys_dist} gives by direct integration with respect to $\theta$ that $f_3=-{C_1 \over 2}\theta^2-C_2\theta + C +f(x,y)$ for some constant $C$ and function $f$. Then the second equation of \eqref{sys_dist} gives by direct integration with respect to $x$ that $f_3=-{C_1 \over 2}(a^2+x^2)-C_2\theta+C_3x+C_4$ and the derivation with respect to $y$ does not appear, because $f_1$ does not depend on $\theta$. Then substituting $f_i$ into $v$ and independent choice of constants gives infinitesimal automorphisms $t_i$.

Then \cite[Lemma 7.23.]{ABB} states that a diffeomorphism on a Lie group is a right translation if and only if it commutes with all left translations and one can check directly that this is the case of $t_i$, $i=1,2,3$.  The infinitesimal symmetry $t_0$ clearly preserves the origin.
\end{proof}

In the Listing \ref{cod_inf} we find infinitesimal symmetries of the control system and the isotropy subalgebra of the origin.   

\tiny 
\begin{lstlisting}[language=Mathematica ,caption={Infinitesimal symetries of the metric}, label={cod_inf}]
ChangeFrame(N);
inf_met := InfinitesimalSymmetriesOfGeometricObjectFields([[n1, n2], r], 
output = "list");
inf_met_stab := IsotropySubalgebra(inf_met, [x = 0, y = 0, theta = 0]);
\end{lstlisting}
\normalsize

In the Listing  \ref{cod_pde} we  find the system of pde's and its solutions directly.
\tiny 
\begin{lstlisting}[language=Mathematica ,caption={Corresponding pde}, 
label={cod_pde}]
v := evalDG(f1(x, y, theta)*n1 + f2(x, y, theta)*n2 + f3(x, y, theta)*n3);
pfh := op(Annihilator([n1, n2]));
r1 := ContractIndices(LieDerivative(v, n1), pfh, [[1, 1]]);
r2 := ContractIndices(LieDerivative(v, n2), pfh, [[1, 1]]);
r3 := DGinfo(LieDerivative(v, r), "CoefficientSet");
r4 := {r1, r2, op(r3)};
pdsolve(r4);
\end{lstlisting}
\normalsize
Let us note that the contact manifold $(N,\N)$ has infinitely  dimensional algebra of infinitesimal symmetries (over $\R$). We can parametrize them by solving the system \eqref{sys_dist}. Clearly, the function $f_3$ can be arbitrary and then $f_2=-{\del f_3 \over \del \theta}$ and $f_1=\theta {\del f_3 \over \del y}+{\del f_3 \over \del x}.$ In particular, fields $t_i$ correspond to choices of $f_3$ of the form $f_3=-\frac{1}{2}(x^2+\theta^2)$ for $t_0$, $f_3=- \theta$ for $t_1$,
$f_3=x$ for $t_2$ and $f_3=1$ for $t_3$.

\begin{rem} Using the same methods we can also find the symmetries of the  system $(K,\mathcal D, k )$. It turns out that  there are only translations corresponding to the right--invariant fields on the  Lie algebra $\mathfrak k$, see Section \ref{pmp-orig}. The natural generators of the symmetry algebra have the expected form
$-y \partial_x + x \partial_y + \partial_{\theta}, \partial_y$ and $\partial_x$.   
\end{rem}

\subsection{Action of translations and changes of initial conditions}
Since translations $t_i$, $i=1,2,3$ correspond to right-invariant fields, their action preserves left-invariant objects on $N$ with respect to the group multiplication. Apart from the fact that they preserve $\N$ and $r$, they also preserve vertical coordinates $h_i$ and normal Hamiltonian $H$ and corresponding Hamiltonian field. Since local extremals are flows of this field, translations $t_i$ preserve vertical part of the solution. On the other hand, the action of $t_i$ corresponds to left multiplication in $N$, so it maps horizontal solutions to horizontal solutions.

\begin{prop*} 
The action of the flow of infinitesimal transformation $t_b=b_1t_1+b_2t_2+b_3t_3$, $b_1,b_2,b_3 \in \R$, at time $s$ maps a local extremal corresponding to a local minimizer  with initial condition $x(0)=y(0)=\theta(0)=0$ to another local extremal with the same  vertical part,  and the horizontal part corresponds to local minimizer with initial condition $x(0)=b_1s$, $y(0)={1 \over 2} b_2b_1 s^2+b_3s$, $\theta(0)=b_2s$.
\end{prop*}
\begin{proof}
The statement follows from the general theory. We can also compute explicitly the action. The transformation takes the form $x \mapsto b_1+x$, $y \mapsto {1 \over 2} b_2 b_1 s^2+b_2xs+b_3s+y$, $\theta \mapsto b_2s+\theta$. Its action maps a local minimizer \eqref{sol_nil} with initial condition $x(0)=y(0)=\theta(0)=0$ to a local minimizer
\begin{align} \label{sol_nil-posunuta}
\begin{split}
x &= b_1s+{1 \over C_1} (C_3-C_2\sin(C_1t)-C_3\cos(C_1t)),\\
y &= 
{1 \over 2} b_2 b_1 s^2+b_2s{1 \over C_1} (C_3-C_2\sin(C_1t)-C_3\cos(C_1t))+b_3\\
&+
{1 \over 4C_1}(
2C_1(C_2^2+C_3^2)t-4C_2C_3\cos(C_1t)+2C_2C_3\cos(2C_1t)\\
&-4C_2^2\sin(C_1t)+(C_2^2-C_3^2)\sin(2C_1t)+2C_2C_3 ),\\
\theta &= b_2s+{1 \over C_1}(C_2 -C_2\cos(C_1t)+C_3\sin(C_1t)).
\end{split}
\end{align}
One can check by direct computation that \eqref{sol_nil-posunuta} together with 
\eqref{vert_sol} solve the system (\ref{ver_nil},\ref{hor_nil}) 
with the initial condition above. Analogous statement holds for the degenerate solutions.
\end{proof}
One can easily see that every point of a suitable neighbourhood of $(0,0,0)$ can be expressed by suitable choice of $b_i$. 

In the Listing \ref{cod_flow} we realize the action of general translation as a transformation.
\tiny 
\begin{lstlisting}[language=Mathematica ,caption={Flows}, label={cod_flow}]
t1 := D_x;
t2 := evalDG(D_y*x + D_theta);
t3 := D_y;
fl := Flow(evalDG(b1*t1 + b2*t2 + b3*t3), s);
tb := Transformation(N, N, ApplyTransformation(fl, [x = x, y = y, theta = theta]));
#The last row changes the data structure to the transformation
\end{lstlisting}
\normalsize
One can apply the transformation $t_b$ on the horizontal solution of the system at the origin. 

\subsection{Action of isotropy symmetries} \label{cut-poin}
The symmetry $t_0=\theta\del_x+{\theta^2-x^2 \over 2}\del_y-x\del_\theta$ generates one-dimensional isotropy subalgebra of all infinitesimal symmetries at $(0,0,0) \in N$.  Its flow $\Fl_{t_0}$ gives a one--parametric family of transformations parameterized by $s$ of the form $x \mapsto \theta \sin(s)+x\cos(s)$, $y \mapsto {\theta^2 - x^2 \over 2}\sin(s)\cos(s)-x\theta\sin^2(s)+y$, $\theta \mapsto \theta \cos(s)-x\sin(s)$.

The general principle states that local minimizers are locally optimal and on each local minimizers, there is a point after which it is not optimal. The first point with this property is called cut-point  \cite[Definition 8.71.]{ ABB}. In particular, if there are two local minimizers starting at the origin which intersect at some point $n \in N$ at the same time, the local minimizer cannot be optimal after this point. If there is an isotropic symmetry with a fixed point $n \neq o$ and a local minimizer going from $o=(0,0,0)$ to $n$, then either the local minimizer is contained in the fixed point set of the symmetry or the symmetry maps the local minimizer to another local minimizer from $o$ to $n$ of the same length, so the above principles apply. 

We use here the symmetry $t_0$ to recover the results for 
non--degenerate normal extremals. The fixed-point set of $t_0$ equals to $S=\{(0,y,0): y \in \R\}$.
\begin{prop*}
Each local minimizer $c(t)$ of the form \eqref{sol_nil} such that $C_1 \neq 0$ and $C_2C_3 \neq 0$ intersects with fixed point set $S$ at the point $(0,{\pi(C_2^2 + C_3^2) \over C_1^2},0)$ at the time $t={2\pi \over C_1}$ and it is the first point with this property. For each $c(t)$, there is a one--parametric family of local minimizers parametrized by $s$  which is given as orbit  of $c(t)$ with respect to the action of the flow of $t_0$.
\end{prop*}
\begin{proof}
A local minimizer  \eqref{sol_nil} intersect with $S$ if and only if $x=\theta=0$, i.e.
$$
\begin{pmatrix} C_3 \\ C_2 \end{pmatrix}
= 
\begin{pmatrix} \sin(C_1 t)  & \cos(C_1 t) \\ 
\cos(C_1 t) & -\sin(C_1 t) \end{pmatrix}
\begin{pmatrix} C_2 \\ C_3 \end{pmatrix},
$$
and thus $C_1 t=2 k \pi$, $k \in \mathbb Z$.  Starting at $t=0$ the first such non trivial point corresponds to $k=1$ in positive direction. The evaluation of $y$ at the time $t={2\pi \over C_1}$ gives mentioned value.

A direct computation shows that $\Fl_{t_0}$ maps a local minimizer \eqref{sol_nil} to a local minimizer $\hat c(t)=( x, y, \theta) $ such that 
$\hat c(0) = (0,0,0)$, $\hat c({2 \pi \over C_1}) = (0,{\pi(C_2^2 + C_3^2) \over C_1^2},0)$ and 
\begin{align}
h_1 &= \dot x =\frac{-C_2 C_1 \cos(C_1 t + s) + C_3 C_1 \sin(C_1 t + s)}{C_1},\\
h_2 &= \dot \theta =\frac{ C_3 C_1 \cos(C_1 t + s) + C_2 C_1 \sin(C_1 t + s)}{C_1}
\end{align}
which satisfies $h_1^2+h_2^2=C_2^2 + C_3^2$, i.e. carries the same parametrisation. 
\end{proof}
Let us note that for the case $C_2=C_3=0$, the solution is degenerate 
and it is contained in $S$ at all times.
\begin{cor*} \label{cut} 
Each local minimizer \eqref{sol_nil} parametrized by arc-length intersects with $S$ at the point $(0,{\pi \over C_1^2},0)$ at the time $t={2\pi \over C_1}$, and it is a first point with this property.
\end{cor*}

This observation reflects known results about the conjugate locus developed by different methods in \cite{ABB,mope}.

In the Listing \ref{cod_trans} we realize the action of the isotropy symmetry $t_0$ as a transformation.  
\tiny 
\begin{lstlisting}[language=Mathematica ,caption={Flows}, label={cod_trans}]
t0 := evalDG(theta*D_x + (theta^2/2 - x^2/2)*D_y - x*D_theta);
f2 := simplify(Flow(t0, s));
t2 := simplify(Transformation(N, N, ApplyTransformation(f2, 
[x = x, y = y, theta = theta])));
\end{lstlisting}
\normalsize

Let us display in  Figure \ref{bomba} corresponding family of local minimizers for the local minimizer from  Example \ref{ex1}, where the red line displays exactly the local minimizer from  Example \ref{ex1}.
\begin{figure}[h!] \label{bomba}
\includegraphics[width=0.3\textwidth]{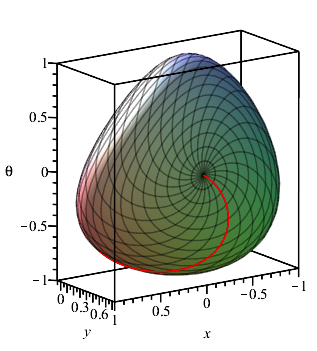} 
\caption{One parametric family of local minimizers}
\end{figure}

\section{Metrics and Lagrangean contact structures} \label{sec_metr}

In the previous section we always considered the fixed sub--Riemannian metric $r$ defined on the distribution $\N$ on $N$.
The choice of the metric $r$ corresponds to the choice of the basis of $\N$ where we consider generators $n_1 =\partial_{\theta} $ and $ n_2 =\partial_x + \theta \partial_y$. The filtered manifold $(N,\N)$ approximates the configuration space of the vertical rolling disc and the generators $n_1$, $n_2$ naturally correspond to the angular velocity and plane velocity. Thus the choice of $n_1$ and $n_2$ gives the choice of units of the two velocities which then impacts the metric and local optimal control. Then, one can choose arbitrary multiples of the fields to get different units of the two velocities. The problem of ratio of the two velocities appears e.g. in the inverse kinematic,  \cite{doty}. It is a known fact that all sub--Riemannian metrics are equivalent on the Heisenberg group,  \cite{ABB}. However, the fields $n_1$ and $n_2$ generate distinguished directions in $\N$, which then give preferred choices of metrics. Thus it is natural to focus on the situation where we still distinguish angular and plane velocity but we do not fix the units. This leads us to so called \emph{Lagrangean contact structures}.

\subsection{Lagrangian contact structure on $(N,\N)$}
A Lagrangian contact structure on a contact manifold is a decomposition of the contact distribution into two isotropic subbundles of the same dimension,  \cite[Section 4.2.3]{parabook}.
On the contact manifold $(N,\N)$ we consider the natural decomposition of $\N$ as $\N=E+F$ for $E=\langle n_1\rangle$ and $F=\langle n_2\rangle$. 
The Lagrangian contact structure $\N=E+F$ is left-invariant with respect to group structure on $N$ from Lemma \ref{grupa}.

 Since the sub--Riemannian metric $r$ is given such that $n_1,n_2$ are orthonormal, the Lagrangian contact structure generalizes $r$ in such a way that we forget the length of vectors and consider only subspaces generated by them. Different choice of generators of $E$ and $F$ then gives different metric which then gives the same orthogonality but differs in the choice of units. In this way we get a class of  sub--Riemannian metrics on $(N,\N)$.

The decomposition can be also viewed as the null space of the sub--Riemannian pseudo-metric on $\N$ of a split signature $(1,1)$ of the form $\text{d}x \text{d} \theta$, respectively the null space of the corresponding conformal class $[\text{d}x \text{d} \theta]$. However, this metric plays no role for the optimal control and can be used only as a different description of the geometric structure.

\subsection{Symmetries of Lagrangian contact structure}
The symmetry algebra of a three-dimensional Lagrangean contact structure is always finite-dimensional and it reaches the maximal possible dimension if the geometry is locally equivalent to the flag manifold corresponding to $\frak{sl}(3,\R)/\fp$, where $\fp$ is a Borel subalgebra. 
In such case, the symmetry dimension equals to $8$ and the geometry corresponds to a flat parabolic geometry, \cite[Section 4.2.3]{parabook}. In particular, the symmetry algebra coincides with $\frak{sl}(3,\R)$.
We will see that this is the case of the homogeneous nilpotent approximation $N$. In fact, the nilpotent approximation of a filtered manifold can be always viewed as a suitable representative of corresponding associated grading which is a flat  geometry.

Since the Lagrangian contact structure is left-invariant, the following fact holds trivially.

\begin{lem*} \label{gminus}
Symmetries $t_1 =\del_x$, $
t_2 =x\del_y+\del_\theta$, 
$t_3 = \del_y$ preserve the Lagrangian contact structure $(N,\N=E+F)$.
\end{lem*}
Thus symmetries reflect the nilpotent symbol $[t_1,t_2]=t_3$ giving the contact grading $\fg^{-2}\oplus \fg^{-1}=\langle t_3 \rangle \oplus \langle t_1,t_2 \rangle$. Then we can find all symmetries by the methods of Tanaka prolongation,  \cite{tan,yam,zel}, see also Appendix \ref{tanaka}. 
\begin{prop*} \label{sym-lagr}
Symmetries of the Lagrangean contact structure $(N,\N=E+F)$ form the Lie algebra $\frak{sl}(3,\R)$ generated by $t_1, t_2, t_3$ and
\begin{align*}
&t_4=x\del_x+2y\del_y+\theta \del_\theta,\\
&t_5=\theta\del_\theta-x\del_x,\\
&t_6={y \over 2} \del_x-{\theta^2 \over 2} \del_\theta,\\
&t_7=x^2\del_x+xy\del_y-(\theta x-y)\del_\theta, \\
&t_8={xy \over 2} \del_x+{y^2 \over 2} \del_y+{\theta \over 2}(y-\theta x)\del_\theta.
\end{align*}
The Lie algebra $\frak{sl}(3,\R)$ carries a contact grading $\fg^{-2} \oplus \fg^{-1} \oplus \fg^{0} \oplus \fg^{1} \oplus \fg^{2}$ as follows: $\fg^{-2}=\langle t_3 \rangle$, $\fg^{-1}=\langle t_1,t_2 \rangle$, $\fg^{0}=\langle t_4,t_5 \rangle$, $\fg^{1}=\langle t_6,t_7 \rangle$ and $\fg^{2}=\langle t_8 \rangle$.
\end{prop*}
\begin{proof}
One can check by direct computation that the fields $t_i,i=1,\dots,8$ satisfy  $\mathcal{L}_{t_i}E \subset E$ and $\mathcal{L}_{t_i}F \subset F$ for the Lie derivative $\mathcal{L}$, i.e. preserve the Lagrangian contact structure, and simultaneously 
satisfy the commutation relations of $\frak{sl}(3,\R)$. 

The fields $t_i$ determine the following gradation: Translations 
$t_1,t_2,t_3$ preserve $\N$, $E=\langle n_1 \rangle$ and $F=\langle n_2 
\rangle$ due to Lemma \ref{gminus} and reflect the grading $\fg^{-2} \oplus 
\fg^{-1} $. The subalgebra $\fg^0$ always contains the grading element,  \cite[Section 3.1.2]{parabook}. In our case the grading element is the element $t_4$ which satisfies
$[t_4,X]=kX$ for $X \in \fg^k$. 
Moreover, $t_5 \in \fg^0$ is 
orthogonal  to $t_4$ with respect to Killing form
and satisfies
$[t_5,X]=X$ for $X \in E$,  
$[t_5,X]=-X$ for $X \in F$ and 
$[t_5,X]=0$ for $X \in \fg^{-2}$.
In particular, the element $t_4$ distinguishes $\fg^{-1}$ from $\fg^{-2}$, and $t_5$ distinguishes $E$ from $F$.
These two elements $t_4,t_5$ form the whole $\fg^0$ for dimensional 
reasons. Indeed, $\fg^0$ equals to its center which has dimension $2$ in the 
three-dimensional case. Finally, $\fg^1 \simeq (\fg^{-1})^*$, $\fg^2\simeq (\fg^{-2})^*$  with respect to Killing 
form, where $t_6$ corresponds to $t_2$, $t_7$ corresponds to $t_1$ and $t_8$ 
corresponds to $t_3$.
Details from the structure theory can be found in 
\cite{parabook}.
The explicit computation of the fields using prolongation procedure is given in Appendix \ref{tanaka}. 
\end{proof}

In Listing \ref{cod_lagr} we find infinitesimal symmetries of the Langrangean contact structure and discuss its properties.
\tiny 
\begin{lstlisting}[language=Mathematica ,caption={Infinitesimal symetries of Lagrangean contact structures}, label={cod_lagr}]
  inf_lagr := InfinitesimalSymmetriesOfGeometricObjectFields([[N1], [N2]], 
  output = "list");
  ChangeFrame(N);
  inf_lagr := InfinitesimalSymmetriesOfGeometricObjectFields([[n1], [n2]],
  output = "list");
  alg_lagr := LieAlgebraData(inf_lagr, al);
  DGsetup(alg_lagr);
  Query("Semisimple");
  Query([e1, e2, e3, e4, e5, e6, e7, e8], "SplitForm");
\end{lstlisting}
\normalsize
The next observation follows from Propositions \ref{sym_met} and \ref{sym-lagr}.
\begin{cor*}
The only transformations that preserve both the Lagranigan contact structure $\N=E+F$ and the metric $r$ are the translations $ \langle t_1, t_2, t_3 \rangle$. Thus there is no transformation preserving both the Lagranigan contact structure $\N=E+F$ and simultaneously the metric $r$ with a fixed point. 
\end{cor*}
 We focus on the action of $t_i$, $i=4,\dots, 8$ on metrics in the next sections.

\subsection{Sub--Riemannian metrics} 
The general principle states that for each transitive group $L$ acting on $N$,  all $L$-invariant objects on $N$ are in one-to-one correspondence with tensors on the tangent space $T_oN=\fm$, which are invariant 
under the isotropy representation, \cite{parabook}. 
Let us now focus on invariant sub--Riemannian metrics on $(N,\N)$.

On the distribution with symbol algebra $\mathfrak m$ one can compute the first step of the prolongation which is a Lie algebra $\mathfrak m \oplus \mathfrak g^{0}$, where the Lie algebra $\mathfrak g^{0}$ is the algebra of all derivations of $\mathfrak m$ which preserve the grading.
Since $\fg^{-1}$ generates $\mathfrak m$, the algebra $\mathfrak g^{0}$ in fact is the algebra of all derivations of $\fg^{-1}$ and  the subspace $\mathfrak m \oplus \mathfrak g^{0}$ is endowed with the natural structure of a graded Lie algebra. In particular, 
$[A,v] = A(v)$ for $A \in \mathfrak g^{0}$ and $v \in \mathfrak m $. 
This means that the action of automorphisms of $\mathfrak g^{0}$ on $\fg^{-1}$ is exactly the adjoint action. 
Each sub--Riemannian metric is given by the reduction corresponding to  
$\mathfrak{so}(2) \hookrightarrow  \text{ad} (\fg^0)  |_{\fg^{-1}} $
and $ \text{ad} (\fg^0)  |_{\fg^{-1}} =  \mathfrak{gl} (2,\mathbb R)$ 
in the Heisenberg case.
Then  every inclusion $\mathfrak{so}(2) \hookrightarrow  \mathfrak{gl} (2,\mathbb R) $ gives an invariant scalar product on $\fg^{-1}$ and 
thus an invariant sub--Riemannian metric on $\N$.  
Let us remark that for each choice of invariant sub--Riemannian metric corresponding to an inclusion of $\mathfrak{so}(2)$, the further prolongation of $\fm \oplus \mathfrak{so}(2) $ is trivial, \cite{ams}.
\begin{lem*}
On the Lagrangean contact structure $(N,\N=E+F)$ there is no 
symmetric, positive definite quadratic form $B$ on $\fg^{-1}$, i.e.
there is no sub--Riemanian metric on $\N$ invariant with respect to symmetries of Lagrangean contact structure.
\end{lem*}
\begin{proof} The Lie algebra $\mathfrak g^0$ generated by 
$t_4$ and $t_5$ from Proposition \ref{sym-lagr} acts on $\mathfrak g^{-1}$ as inner 
derivation by matrices:
\begin{align*} 
a_1:=\text{ad}({t_4}) |_{\fg^{-1}} =\begin{pmatrix} -1 & 0 \\ 0 & -1 \end{pmatrix} , \: \:
a_2:=\text{ad}({t_5}) |_{\fg^{-1}} =\begin{pmatrix} 1 & 0 \\ 0 & -1 \end{pmatrix}.
\end{align*} 
Each inclusion $\mathfrak{so}(2) \hookrightarrow  \mathfrak{gl} (2,\mathbb R) $ is given by a symmetric, positive definite matrix $B$ as 
$$\mathfrak{so}(2,B)  = \{A \in \text{Mat}( 2,\mathbb R) : AB+ BA^t =0 \}.$$ 
The element  $ \alpha_1 a_1 + \alpha_2 a_2$
belongs to $\mathfrak{so}(2,B)$ only if $B=0$. This is contradiction 
with the fact that $B$ is positive definite. In other words,   $\mathfrak{so}(2,B) \cap \langle a_1,a_2 \rangle = \emptyset$.
\end{proof}
In Listing \ref{cod_a} we compute the adjoint action of $\fg^0$ on $\fg^{-1}$. 

\tiny 
\begin{lstlisting}[language=Mathematica ,caption={Metrics}, label={cod_a}]
rep := LieAlgebras:-Adjoint(alg_lagr);
a1 := SubMatrix(rep[4], [2, 3], [2, 3]);
a2 := SubMatrix(rep[5], [2, 3], [2, 3]);
#The order of infinitesimal symetries can differ 
#in various versions of Maple and our choice reflects Maple 2019.
\end{lstlisting}
\normalsize

\subsection{Action of symmetries of Lagrangean contact structure on metrics} 
Although there is no sub--Riemannian metric which is invariant with respect to the action of symmetries of Lagrangian contact structure, we can still study action of these transformations on the metrics and thus on corresponding control systems. In fact, the action of such transformation maps distinguished generators $n_1, n_2$ of the distribution to another generators which are functional multiples of $n_1, n_2$. These news generators give different units and ratio of angular and plane velocity and thus a different metric. However, we can use the action of the transformations to compare the control system and corresponding local extremals for such different units. Let us demonstrate this principle on  symmetries from Proposition \ref{sym-lagr}.

Let us consider the transformation $t_6$ from Proposition \ref{sym-lagr}. Its flow gives a one-parameter family of transformations $f_s$ parametrized by $s$ of the form $x \mapsto \frac{sy}{2}+x$, $y \mapsto y$, $\theta \mapsto \frac{2\theta}{w}$, where $w=s\theta+2$. 
The action of $f_s$ on the vector fields $n_i$ is as follows:
\begin{align*}
f_s^*n_1 & = \frac{w^2}{4} \partial_{\theta} = \frac{w^2}{4} n_1,\\
f_s^*n_2 & = - \frac{2}{w} \partial_{x} 
- \frac{2 \theta}{w} \partial_{y} =  - \frac{2}{w}n_2, \\
f_s^*n_3 &= [f_s^*n_1,f_s^*n_2] = \frac{s}{2} \partial_{x} +  \partial_{y}.
\end{align*}
For each $s$, the fields $f_s^*n_1,f_s^*n_2$ generate $\N$. These fields together with $f_s^*n_3$ generate a Heisenberg Lie algebra and determine corresponding multiplication structure on $N$ for which are the fields left-invariant. 
The canonical metric $\tau_{s}=f_s^* r$ for this structure has the form
\begin{align*}
\tau_{s}=
(\text{d}x- \frac{s}{2}  \text{d}y )^2
+\left( \frac{4}{w^2} \text{d}\theta \right)^2. 
\end{align*}
Indeed, the expressions in brackets are the first two forms in the dual basis of $f^*_sn_i$.
Since we solve local problem in a neighbourhood of the origin, we simply exclude singularities.

Altogether, we get a one-parametric family of left-invariant control systems $(N,\N,\tau_s)$ (possibly for different multiplications) and the choice $s=0$ corresponds to identity and gives the original control system. Since  $[f^*_s n_1, f^*_s n_2]=f^*_s[n_1, n_2]=f^*_s n_3$, the vertical system has the shape  \eqref{ver_nil}(although the vertical coordinate functions differ according to the multiplication structure).
The horizontal system simply says that the curve is contact to the distribution but we shall use new vector fields, i.e. we have 
$\dot c(t)=\frac{u_1 w^2}{4}n_1-\frac{2u_2}{w}n_2$.
This can be viewed (locally) as a change of coordinates in $N$.
It follows from the functoriality that the action of $f_s$ maps local minimizers \eqref{sol_nil} onto minimizers of this new system. We present in Figure  \ref{f1a}
images of such minimizer for the  choice of constants $C_1=2,C_2=-{\sqrt{2} \over 2}, C_3={1 \over 2}, C_4=C_5=1$ for $s=0$ (red), $s=1$ (green), $s=3$ (blue). 
For better mechanical illustration we present the angular and plane movements in Figures  \ref{f1b} and \ref{f1c}.

\begin{figure}[h] 
	\subfigure[Local mimimizer]{\includegraphics[width=0.26\textwidth]{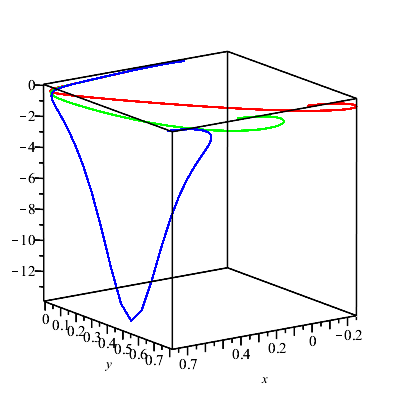}\label{f1a}}
	\subfigure[The parameter $\theta$]{\includegraphics[width=0.24\textwidth]{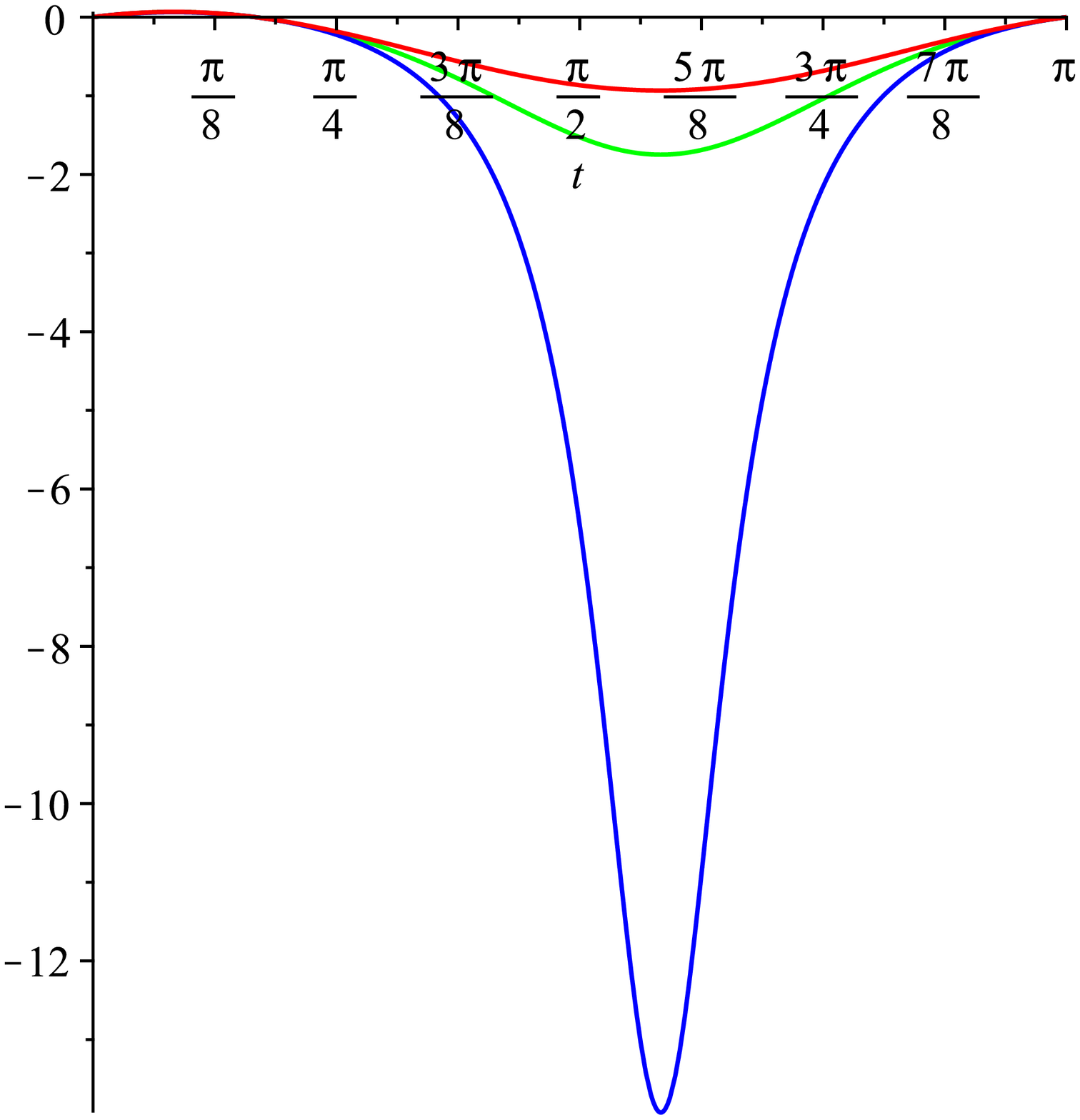}\label{f1b}}
	\subfigure[Trajectory of the contact pont in the plane $(x,y)$]{\includegraphics[width=0.24\textwidth]{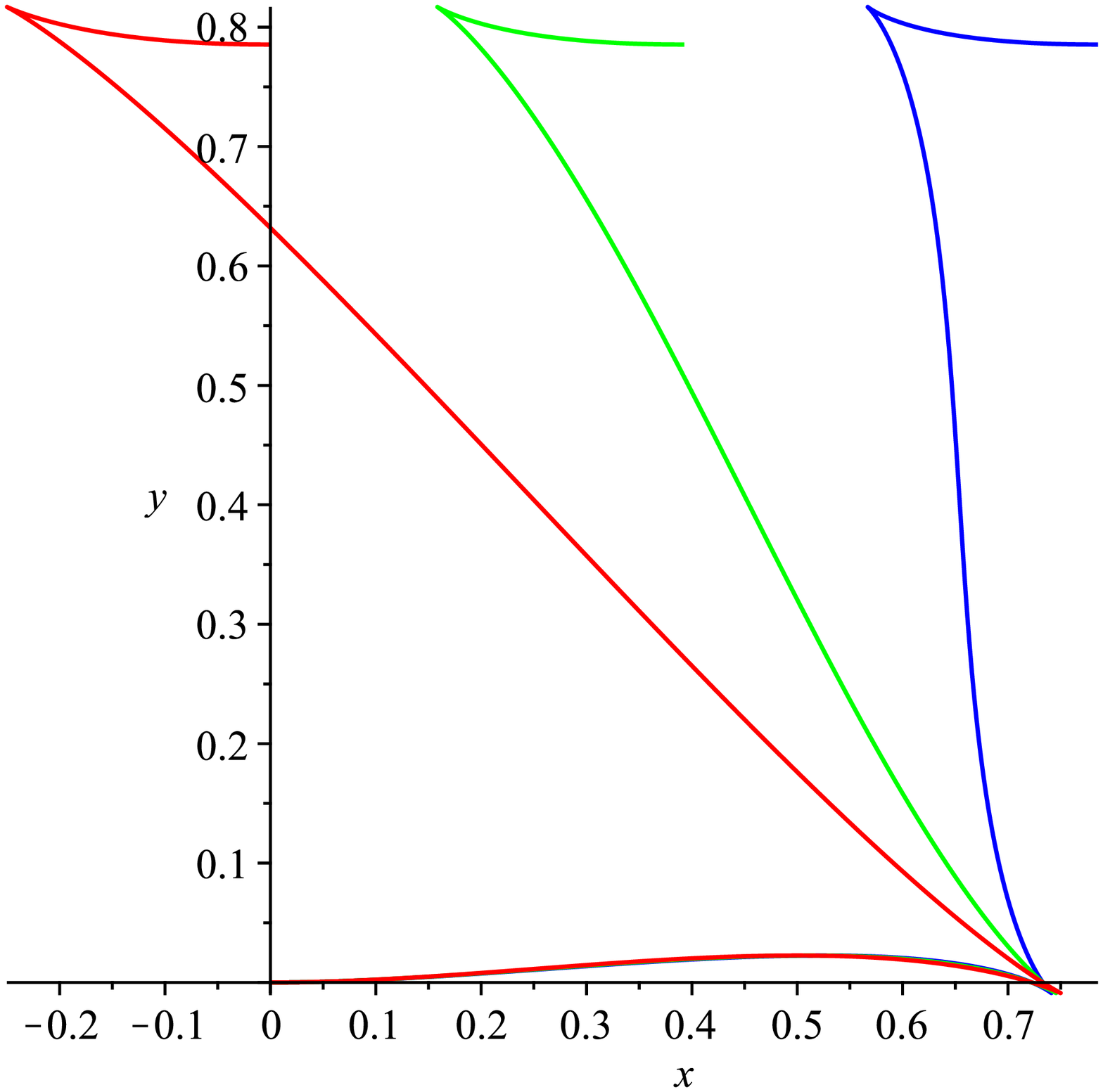} \label{f1c}}
	\caption{Action of one--parametric family for $t_6$}
	\label{geod_6} 
\end{figure}

In  Listing \ref{cod_gra}, we check by direct computation that transformed curves are solutions of the transformed system.
\tiny 
\begin{lstlisting}[language=Mathematica ,caption={Action on solutions}, label={cod_gra}]
psol2mod := solve({rhs(poc2[1]), rhs(poc2[2]), rhs(poc2[3])}, 
{_C1, _C4, _C5, _C6});
sol20mod := simplify(subs(psol2mod, sol2));
#  We describe the solution where in the form where the choice of constants
# is compatible 
t6 := evalDG(y/2*D_x - theta^2/2*D_theta);
f6 := simplify(Flow(t6, s));
p6 := simplify(PushPullTensor(f6, InverseTransformation(f6), [n1, n2]));
evalDG(h1(t)*p6[1] + h2(t)*p6[2]);
GC := subs(x = x(t), y = y(t), theta = theta(t), 
GetComponents(%, [D_x, D_y, D_theta]));
sys6 := {op(sys_vert), diff(theta(t), t) = GC[3], diff(x(t), t) = GC[1],
 diff(y(t), t) = GC[2]};
at6 := simplify(ApplyTransformation(f6, 
[rhs(sol20mod[2]), rhs(sol20mod[3]), rhs(sol20mod[1])]));
test := {op(sol_vert[2]), theta(t) = at6[3], x(t) = at6[1], y(t) = at6[2]};
pdetest(test, sys6);
\end{lstlisting}
\normalsize

One can check by direct computation that the symmetry algebra of the control system $(N,\N,\tau_s)$ is generated by four fields where three of them are the translations of the corresponding left-invariant structure the remaining one generates the one-parametric family of rotations preserving the origin. The corresponding fixed-point set has the form $(\frac{\ell s}{2},\ell,0)$
and this is exactly the image of the fixed-point set $S=\{(0,\ell,0)\}$ of $t_0$. We know from Section \ref{cut-poin} that the solutions \eqref{sol_nil} stop to be optimal at the points where they intersect with the fixed point $(0,\ell,0)$ and these points map exactly to the points where the transformed solutions intersect with the elements of the form $(\frac{\ell s}{2},\ell,0)$ and stop to be optimal.

Finally let us briefly show the behaviour of the systems for the elements $t_4$ and  $t_8$. The action of one parametric family of transformations corresponding to $t_4$ on the metric $r$ gives the family of metrics $\mu_s =e^{2s} \text{d}x^2 +e^{2s} \text{d} \theta^2$.
Let us note that the metrics belong to the conformal class of $r$. 
Applying the ideas for $t_4$, we get minimizers that we display 
in Figure \ref{f4} for the same choice of constants.

\begin{figure}[h] 
	\subfigure[Local mimimizer]{\includegraphics[width=0.24\textwidth]{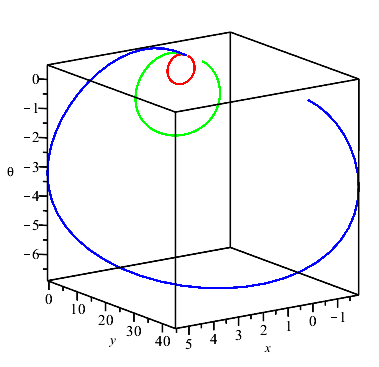}}
	\subfigure[The parameter $\theta$]{\includegraphics[width=0.24\textwidth]{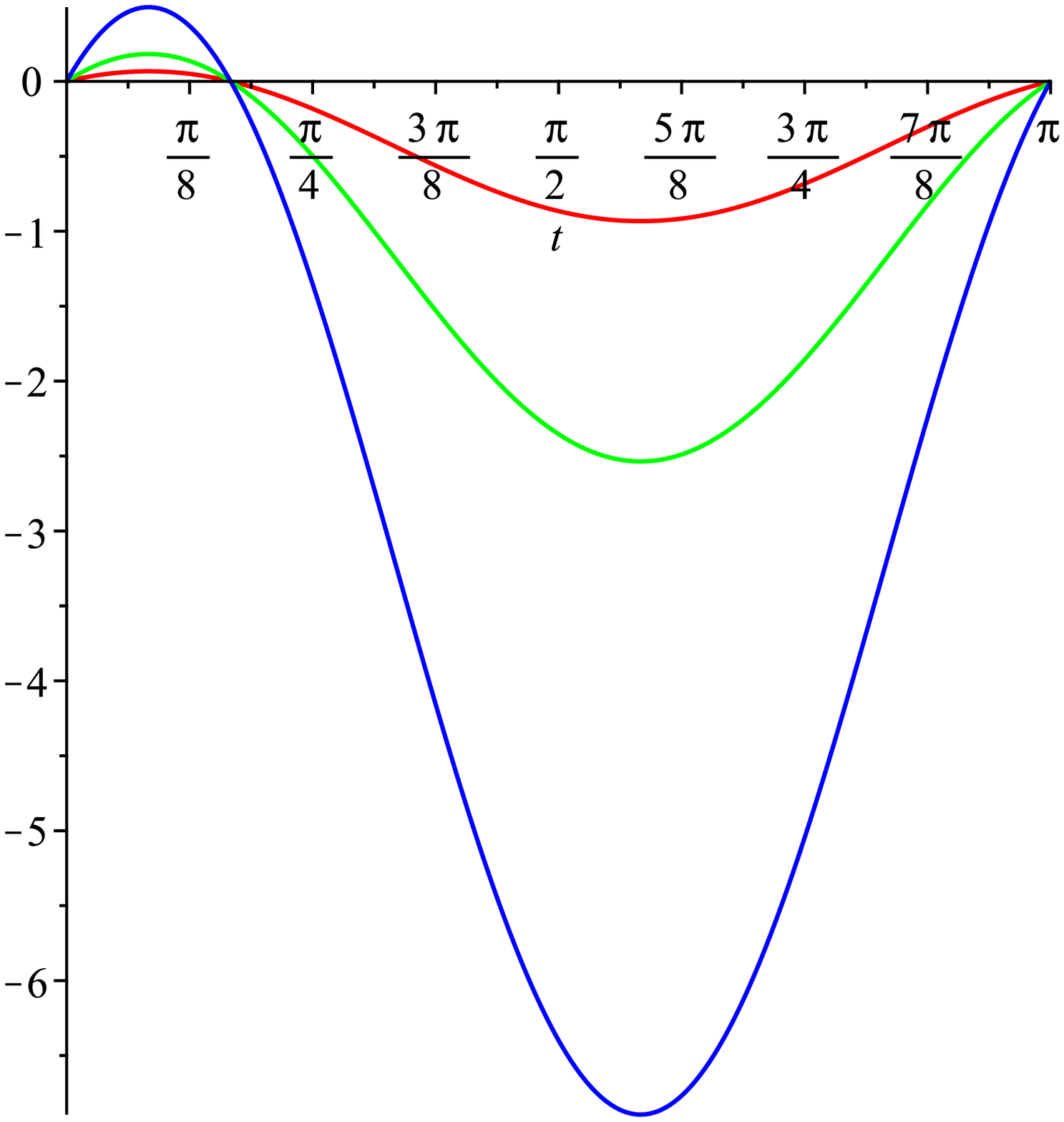}}
	\subfigure[Trajectory of the contact pont in the plane $(x,y)$]{\includegraphics[width=0.24\textwidth]{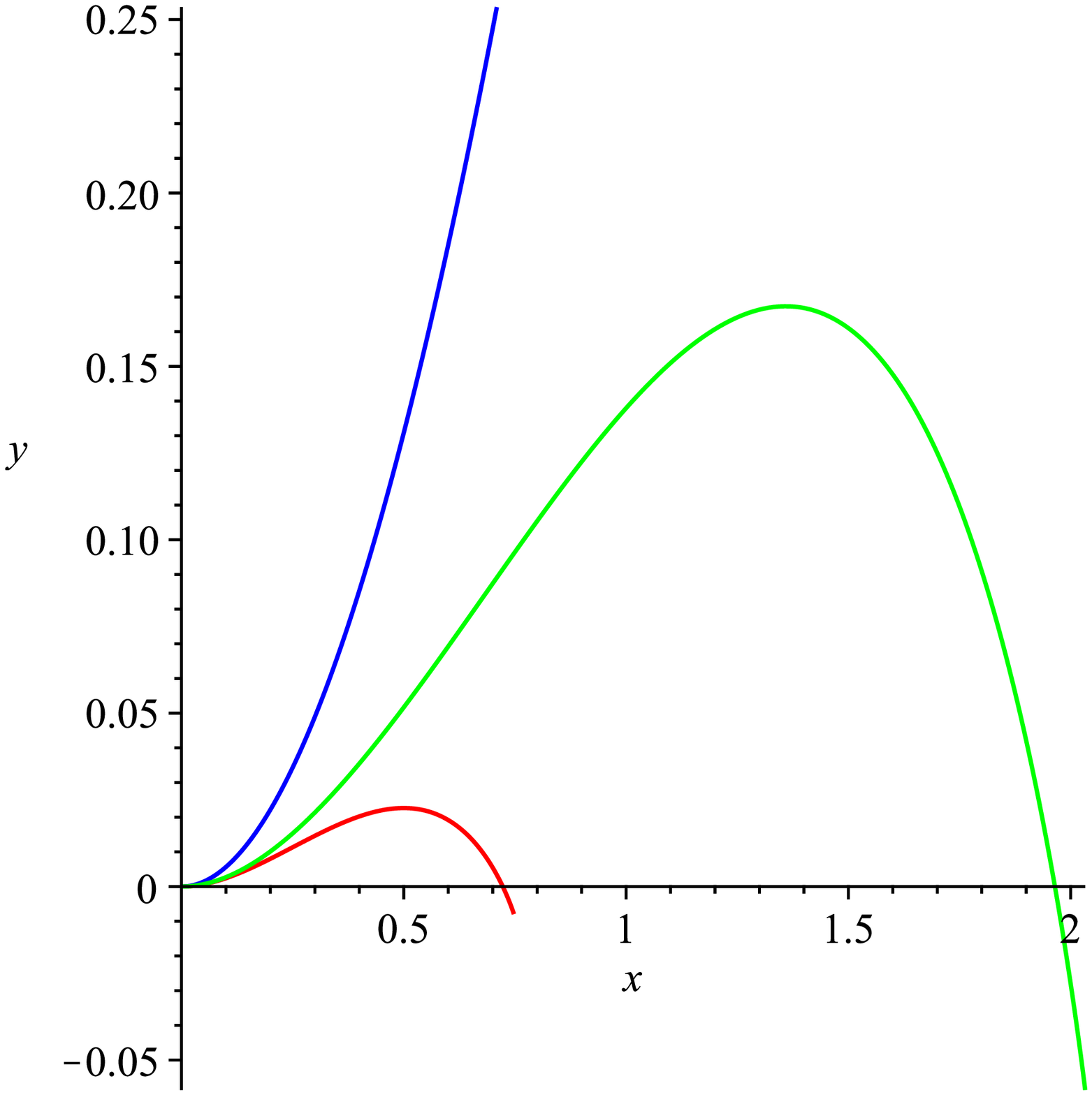} }
	\caption{Action of one--parametric family for $t_4$}
	\label{f4} 
\end{figure}

The computations are generally more technical for transformations from higher parts of grading. We only display the local minimizers corresponding to $t_8$ in Figure \ref{f8} for the same constants.

\begin{figure}[h] 
	\subfigure[Local mimimizer]{\includegraphics[width=0.24\textwidth]{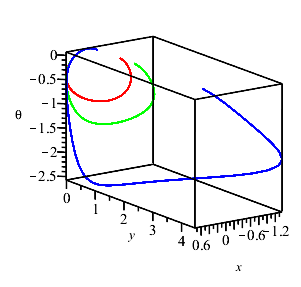}\label{f1a}}
	\subfigure[The parameter $\theta$]{\includegraphics[width=0.24\textwidth]{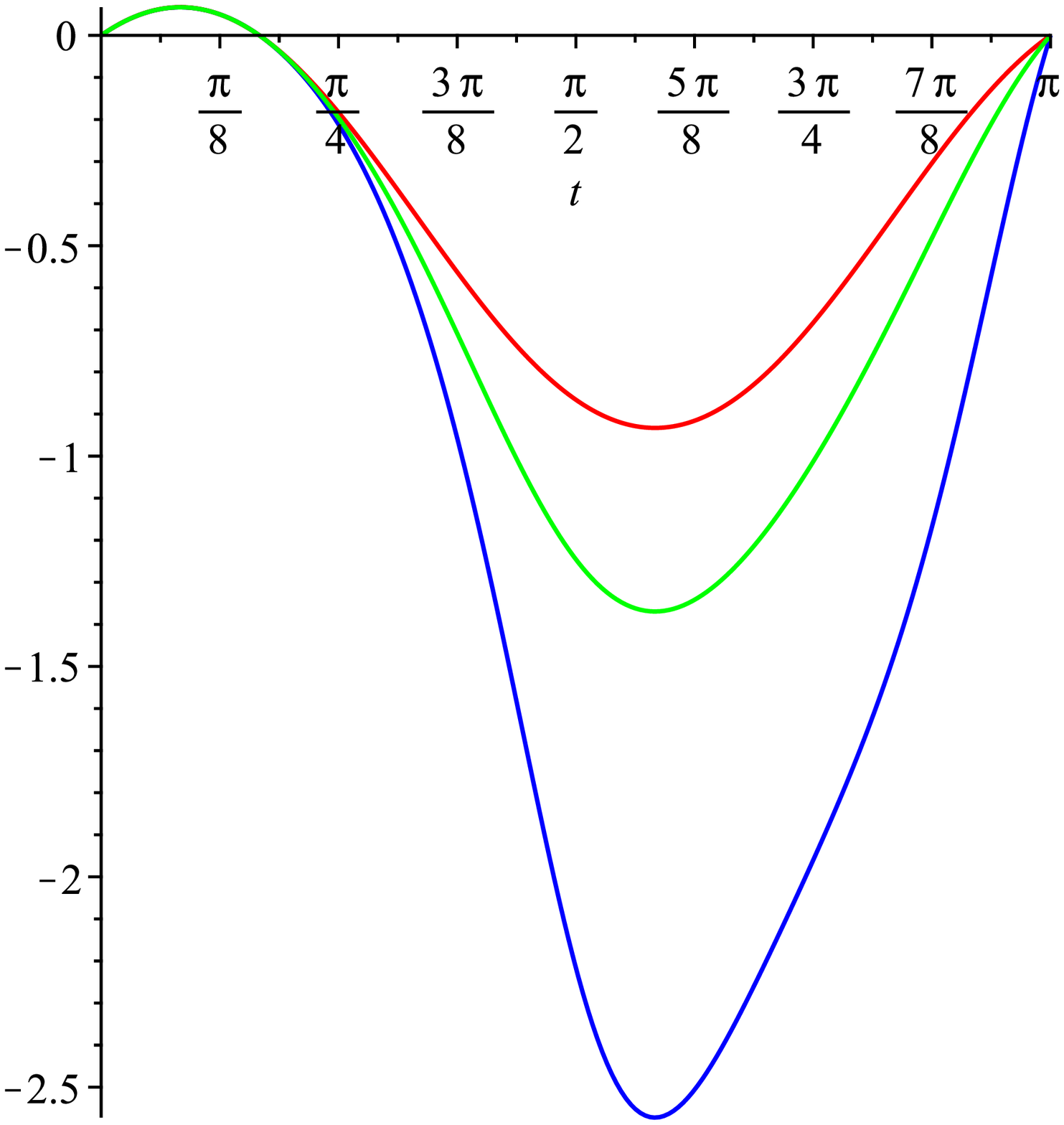}\label{f1b}}
	\subfigure[Trajectory of the contact pont in the plane $(x,y)$]{\includegraphics[width=0.24\textwidth]{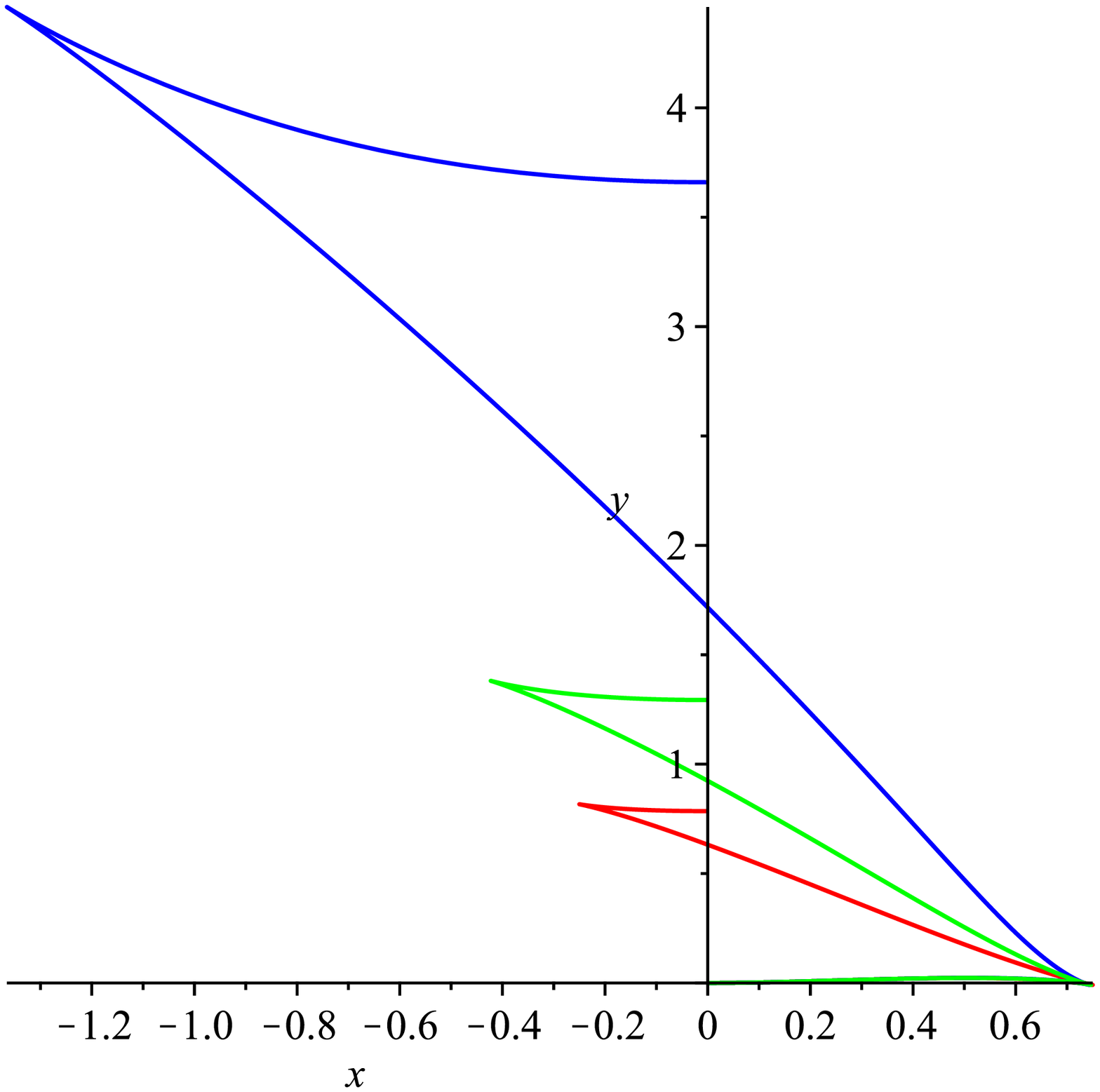} \label{f1c}}
	\caption{Action of one--parametric family for $t_8$}
	\label{f8} 
\end{figure}

Let us emphasize that in general, solving the new system may be much more difficult than applying the transformations. This is e.g. the case of $t_8$. 

\textbf{Acknowledgements.} 
The first two authors were supported by the grant of the Czech Science Foundation no. 17-21360S, "Advances in Snake-like Robot Control" and by the Grant No. FSI-S-17-4464. The last author was supported by the grant of the Czech Science Foundation no. 17-01171S, `Invariant differential operators and their applications in geometric modelling and control theory'. The algebraic computations are calculated in CAS Maple package DifferentialGeometry, \cite{dg}.

\appendix

\section{Tanaka prolongation} \label{tanaka}
The Heisenberg algebra can be seen as graded nilpotent Lie algebra 
$\mathfrak m = \mathfrak g^{-2} \oplus  \mathfrak g^{-1}$, 
where $\mathfrak g^{-2} = \langle n_3 \rangle$ and   $\mathfrak g^{-1} = \langle n_1, n_2 \rangle$. In the first step, we consider $\mathfrak{der} \equiv \text{Hom}
(\mathfrak g^{-1}, \mathfrak g^{-1})$ as the full algebra of derivations of $\mathfrak m $ preserving the grading. 
One can construct the algebraic universal (Tanaka) prolongation of the graded Lie algebra  $\mathfrak m \oplus \mathfrak g^0$ for $ \fg^0 \subset\mathfrak{der}$,  \cite{tan,zel}. 
For simplicity we consider the  generators of $\mathfrak{der}$ as
\begin{align*}
\Lambda^0_1 (n_1) = n_1, \: \:   \Lambda^0_1 (n_2) = n_2, \: \: \: \: \: \: \:
\Lambda^0_2 (n_1) = n_1, \: \:   \Lambda^0_2 (n_2) =  - n_2, \\
\Lambda^0_3 (n_1) = n_2, \: \:   \Lambda^0_3 (n_2) = n_1,  \: \: \: \: \: \: \:
\Lambda^0_4 (n_1) = n_2, \: \:   \Lambda^0_4 (n_2) = - n_1. 
\end{align*}
Additional structures on the distribution can be encoded by subalgebras as reductions of the principal bundle.
 
Let us consider the Lagrangean contact structure from Section \ref{sec_metr}.
Corresponding $\fg^0$ is generated by $\Lambda^0_1$ and $  \Lambda^0_2$ and one can see that $\Lambda^0_1 (n_3) =  2 n_3 $ and 
$ \Lambda^0_2 (n_3) = 0$.
 Since $\mathfrak g^0$ is a subalgebra of the algebra of derivation of $\mathfrak m$ preserving the grading, the subspace $\mathfrak m \oplus \mathfrak g^0$ is endowed with the natural structure of graded Lie algebra such that $[f,v]= f(v),$ where $f \in \mathfrak g^0$ and 
 $v \in \mathfrak m$. 
 Finally, we have $[ \Lambda^0_1 ,\Lambda^0_2] (n_1)  =  [ \Lambda^0_1 ,\Lambda^0_2] (n_2) =0 $. 
 Thus we have the Lie algebra $\mathfrak m \oplus \mathfrak g^0$.
 The next grading  part is 
 \begin{align*}
 \mathfrak g^{1} = \text{Hom} (\mathfrak g^{-2} ,\mathfrak g^{-1}) \cap \text{Hom} (\mathfrak g^{-1} ,\mathfrak g^{0}).
 \end{align*}
 Straightforward computation leads to  
$
\mathfrak g^{1}= \langle \Lambda_1^1, \Lambda_2^1 \rangle, 
$
where
\begin{align*}
\Lambda_1^1 (n_1) &=  \Lambda^0_1 +
3  \Lambda^0_2, \: \: \:
\Lambda_1^1  (n_2) =0, \:\:\:
\Lambda_1^1  (n_3) = - 2  n_2, \\
\Lambda_2^1 (n_1) &=  0, \: \: \:
\Lambda_2^1  (n_2) =\Lambda^0_1 -
3  \Lambda^0_2, \:\:\:
\Lambda_2^1  (n_3) =  2  n_1
\end{align*}
and we get Lie algebra $\mathfrak m \oplus \mathfrak g^0 \oplus \mathfrak g^1 $.

In Listing \ref{cod_tan} we compute the first Tanaka prolongation of $\fm$.
\tiny 
\begin{lstlisting}[language=Mathematica ,caption={Algebraic Tanaka prolongation}, label={cod_tan}]
Bra := [[E2, E3] = E1, [E4, E1] = 2*E1, [E4, E2] = E2,
[E4, E3] = E3, [E5, E1] = 0, [E5, E2] = E2, [E5, E3] = -E3];
Alg := LieAlgebraData(Bra, [E1, E2, E3, E4, E5], A, grading = [-2, -1, -1, 0, 0]);
DGsetup(Alg);
PA := TanakaProlongation(A, 1, AlgPA);
\end{lstlisting}
\normalsize

In the next step, we shall find the space 
$
\mathfrak g^{2} = \text{Hom} (\mathfrak g^{-2} ,\mathfrak g^{0}) \cap \text{Hom} (\mathfrak g^{-1} ,\mathfrak g^{1})
$
 and straightforward computation leads to   
 $\mathfrak g^2 = \langle \Lambda^2 \rangle$, where
\begin{align*}
\Lambda^2 (n_1) &= \Lambda^1_2 , \: \: \:
\Lambda^2  (n_2) =-\Lambda^1_1  , \:\:\:
\Lambda^2  (n_3) = 2 \Lambda^0_1.
\end{align*}
In the next step we shall find the space $
\mathfrak g^{3} = \text{Hom} (\mathfrak g^{-2} ,\mathfrak g^{1}) \cap \text{Hom} (\mathfrak g^{-1} ,\mathfrak g^{2})
$
but this is trivial.  
The algebraic Tanaka prolongation   has the form 
\begin{align*} \mathfrak g :=
\mathfrak g^{-2} \oplus \mathfrak g^{-1} 
\oplus \mathfrak g^{0} 
\oplus \mathfrak g^{1} 
\oplus \mathfrak g^{2}.
\end{align*}
The Lie algebra $ \mathfrak g$  has the structure of $\mathfrak{sl}(3 ,\R)$ and explicit multiplicative structure is in Table \ref{mg}.
\begin{table}[h!] 
 	\begin{center}
 		\begin{tabular}{ c|| c | c | c | c| c |c |c | c|} 
 		&	$e_{1}$ & $e_2$ & $e_3$ &  $e_4$ & $e_5$ &$e_6$ &$e_7$  & $e_8$ \\
 			\hline \hline
 			 $e_{1}$ & $0$ & $0$ & $0$ & $-2 e_1$ & 0&$2e_3$&$-2e_3$&$- 2 e_4$\\
 			 \hline
              $e_2$ & $0$ & $0$ & $e_1$ & $-e_2$  & $-e_2$ &  $-e_4 - 3  e_5$ &0& -$e_7$	\\ 
 			 \hline
 			  $e_3$ & $0$ & $-e_1$ & $0$ & $-e_3$& $e_3$ &0&$-e_4 +
 			   			3  e_5$&$e_6$\\
 			 \hline
 			 $e_4$ & 2$e_1$ & $e_2$ & $e_3$ & $0$ & 0&$-e_6$ &$-e_7$&-$e_8$\\
 			 \hline
 			  $e_5$ & $0$ & $e_2$ & $-e_3$  & $0$ & 0&$-e_6$&$e_7$&0\\
 			\hline
 			$e_6$ & $ -2  e_3$ &$e_4 + 3  e_5$ & $0$  & $e_6 $  & $e_6$ &0 &$2 e_8$&0\\ 
 			\hline 
 			$e_7$ &$2 e_2$& $0$ & $e_4 -3  e_5$  & $e_7$& $- e_7 $ &$2 e_8$ & 0&0\\
 			\hline 
 		 $e_8$	& $ 2 e_4$ & $e_7$ & $-e_6$  &  $e_8$ & 0&0 & 0&0\\
 		\hline
 		\end{tabular}   
 	\end{center}
 	\caption{The graded algebra $\mathfrak g  \cong
 	\langle e_1\rangle  \oplus\langle e_2, e_3\rangle  \oplus\langle e_4,e_5\rangle 
 	\oplus\langle e_6,e_7\rangle  \oplus\langle e_8\rangle  $   \label{mg} }
 \end{table}
In Listing \ref{cod_tan} we compute the second (and thus full) Tanaka prolongation of $\fm$.
\tiny 
\begin{lstlisting}[language=Mathematica ,caption={Second Algebraic Tanaka prolongation}, label={cod_tan}]
Bra := [[E2, E3] = E1, [E4, E1] = 2*E1, [E4, E2] = E2,
[E4, E3] = E3, [E5, E1] = 0, [E5, E2] = E2, [E5, E3] = -E3];
Alg := LieAlgebraData(Bra, [E1, E2, E3, E4, E5], A, grading = [-2, -1, -1, 0, 0]);
DGsetup(Alg);
PA := TanakaProlongation(A, 2, AlgPA);
\end{lstlisting}
\normalsize

Let us now find the corresponding geometric prolongation. The Maurer--Cartan form on the Lie group $N$ is
$$ \omega = (e_3 - \theta e_1) \text{d}x + 
e_1 \text{d}y + e_2 \text{d}{\theta} =
(\text{d}y - \theta \text{d}x ) e_1  +   \text{d}{\theta} e_2 +\text{d}x e_3$$

It follows from Table \ref{mg} that  $e_1 \in \mathfrak g^{-2}$ 
and we look for a vector field $Y_{e_1}$ on $N$ such that 
$\omega(Y_{e_1}) = e_1$. It follows from MC form that 
$Y_{e_1} = \partial_y$.

For $e_2$ from Table \ref{mg} contained in $\fg^{-1}$ we look for a vector field $Y_{e_2}$ on $N$ such that  $\omega(Y_{e_2}) = u^{-1} + u^{-2}$ that satisfy   
$ u^{-1} = e_2 $ and $ \text{d}u^{-2} = [u^{-1},\omega^{-1} ] $.
This means $$\text{d}u^{-2} = [e_2,\text{d}{\theta} e_2 +\text{d}x e_3] = \text{d}x e_1$$ and we get by direct integration $ u^{-2} = x e_1 $. 
Altogether, $\omega(Y_{e_2}) =  e_2 +  x e_1$ and thus $Y_{e_2} = \partial_{\theta} + x \partial_y  $. 
For  $e_3 \in \mathfrak g^{-1}$ we get similarly 
 $\omega(Y_{e_2}) =  e_3 -{\theta} e_1$ and thus $Y_{e_3} = \partial_{x}$.
 Let us emphasize that these are exactly  the right--invariant vector fields from 
Proposition \ref{pravo-inv}.

For $e_4 \in \fg^0$ we look for $Y_{e_4}$ such that $\omega(Y_{e4})=
u^0+u^{-1}+u^{-2}$ satisfying $  u^{0} = e_4$ together with  
$\text{d}u^{-1} = [u^{0},\omega^{-1} ]$ and 
$\text{d}u^{-2} = [u^{-1},\omega^{-1} ] +  [u^{0},\omega^{-2} ]$.
We solve
$$\omega(Y_{e_4}) = e_4  +x e_3 +\theta e_2+(2 y- \theta  x) e_1  =
  e_4+x e_3+\theta e_2  +(2 y- \theta  x) e_1$$ and thus we get $
  Y_{e_4} =  x \partial_x + 2y \partial_y + \theta \partial_{\theta}$.
For $e_5 \in \fg^0$ we similarly have 
$ \omega(Y_{e_5}) = e_5   -x e_3 +\theta e_2 +(\theta  x )e_1 $ and we get   
$   Y_{e_5} =  -x \partial_x + \theta \partial_{\theta}
$.
 
 For $e_6 \in \mathfrak g^1$ we solve $
   u^{1} = e_6 $ and  
$   \text{d}u^{0} = [u^{1},\omega^{-1} ]$ together with 
$   \text{d}u^{-1} = [u^{1},\omega^{-2} ] +  [u^{0},\omega^{-1} ] $ and 
$   \text{d}u^{-2} = [u^{0},\omega^{-2} ] +  [u^{-1},\omega^{-1} ] 
$.

In this way, we construct all remaining fields. Altogether, we get 
  \begin{align*}  
 Y_{e_1} &= \partial_y, \: \: \: \: \: \:\: \: \: \:\: \: \: \:\: \: \: \:
 \: \: \: \:\: \: \: \:\: \:  \:\: \: \: \:Y_{e_5} =  -x \partial_x + \theta \partial_{\theta}, \\
 Y_{e_2} &= \partial_{\theta} + x \partial_y, 
 \: \: \: \:\: \: \: \:\: \: \: \: \:\: \: \: \:\: \: \: \:Y_{e_6} =  -2y \partial_x + 2 \theta^2 \partial_{\theta}, \\
  Y_{e_3} &= \partial_{x},\: \: \: \:\: \: \: \:\: \: \: \:\: \: \: \:\: \: \: \:\: \: \: \:\: \: \: \:\: \: \: \:  \:Y_{e_7} =  2x^2 \partial_x +  2xy \partial_y + (2 y - x \theta), \\
  Y_{e_4} &=  x \partial_x + 2y \partial_y + \theta \partial_{\theta},  
  \: \: \: \:    Y_{e_8} = 2 xy \partial_x + 2 y^2 \partial_y + (2y \theta - 2x \theta^2)
   \partial_{\theta}.
   \end{align*} 
Let us emphasize that these fields coincides with fields $t_i$, $i=1,\dots, 8$ from Proposition \ref{sym-lagr} up to a constant multiplies.

\section{The Maple code} \label{sec_mapl}
We present here the whole Maple code which is prepared for Maple 2019.1, Tuesday, May 21, 2019. The corresponding file disc-listing.mw shall be also available on arXiv.  

\tiny

\begin{verbatim}
restart;
with(DifferentialGeometry);
with(Tools);
with(LieAlgebras);
with(Tensor);
with(PDEtools, casesplit, declare);
with(GroupActions);
with(LinearAlgebra);
DGsetup([x, y, theta], M);
pf := evalDG(dy*cos(theta) - dx*sin(theta));
an := Annihilator([pf]);
X1 := an[1];
X2 := an[2];
X1 := an[1];
X2 := evalDG(sin(theta)*an[2]);

X12 := LieBracket(X1, X2);
Alg := LieAlgebraData([X1, X2, X12]);
DGsetup(Alg);
MultiplicationTable();
Query("Solvable");
Query("Nilpotent");

ChangeFrame(M);
db := DualBasis([X1, X2, X12]);
k := evalDG((db[1] &t db[1]) + (db[2] &t db[2]));

Alg := LieAlgebraData([X1, X2, X12]);
DGsetup(Alg);
MultiplicationTable();
Query("Solvable");
Query("Nilpotent");
ChangeFrame(M);
db := DualBasis([X1, X2, X12]);
m := evalDG((db[1] &t db[1]) + (db[2] &t db[2]));

DGsetup([v1, v2, v3], VS);
DGsetup(Alg);
LieAlgebras:-Adjoint(Alg, representationspace = VS);
LieAlgebras:-Adjoint(e1*h1 + e2*h2);

DGsetup([x, y, theta], N);
n1 := evalDG(D_theta);
n2 := evalDG(D_y*theta + D_x);
n3 := LieBracket(n1, n2);
NAlg := LieAlgebraData([n1, n2, n3]);
DGsetup(NAlg);
Query("Solvable");
Query("Nilpotent");
DGsetup([x1, x2, x3], L);
T := Transformation(L, L, [x1 = x1 + y1, x2 = x2 + y2, x3 = x1*y2 + x3 + y3]);
DGsetup([y1, y2, y3], G);
LG := LieGroup(T, G);
InvariantVectorsAndForms(LG);

ChangeFrame(N);
ndb := DualBasis([n1, n2, n3]);
r := evalDG((ndb[1] &t ndb[1]) + (ndb[2] &t ndb[2]));

DGsetup(NAlg);
LieAlgebras:-Adjoint(NAlg, representationspace = VS);
ADJ := LieAlgebras:-Adjoint(e1*h1(t) + e2*h2(t));
SV := map(diff, Matrix([h1(t), h2(t), h3(t)]), t) - MatrixMatrixMultiply(Matrix([h1(t), h2(t), h3(t)]), ADJ);
sys_vert := {SV[1, 1], SV[1, 2], SV[1, 3]};
sol_vert := pdsolve(sys_vert);
pdetest(sol_vert[1], sys_vert);
pdetest(sol_vert[2], sys_vert);
GC := subs(x = x(t), y = y(t), theta = theta(t), GetComponents(evalDG(h1(t)*n1 + h2(t)*n2), DGinfo("FrameBaseVectors")));
sys := {op(sys_vert), diff(theta(t), t) = GC[3], diff(x(t), t) = GC[1], diff(y(t), t) = GC[2]};
sol := pdsolve(sys);
pdetest(sol[1], sys);
pdetest(sol[2], sys);
sol2 := [sol[2][4], sol[2][5], sol[2][6]];
poc2 := eval(subs(t = 0, sol2));
psol2 := solve({rhs(poc2[1]), rhs(poc2[2]), rhs(poc2[3])});
sol20 := simplify(subs(psol2, sol2));
eval(subs(t = 0, sol20));

ChangeFrame(N);
inf_met := InfinitesimalSymmetriesOfGeometricObjectFields([[n1, n2], r], output = "list");
inf_met_stab := IsotropySubalgebra(inf_met, [x = 0, y = 0, theta = 0]);

v := evalDG(f1(x, y, theta)*n1 + f2(x, y, theta)*n2 + f3(x, y, theta)*n3);
pfh := op(Annihilator([n1, n2]));
r1 := ContractIndices(LieDerivative(v, n1), pfh, [[1, 1]]);
r2 := ContractIndices(LieDerivative(v, n2), pfh, [[1, 1]]);
r3 := DGinfo(LieDerivative(v, r), "CoefficientSet");
r4 := {r1, r2, op(r3)};
pdsolve(r4);

t1 := D_x;
t2 := evalDG(D_y*x + D_theta);
t3 := D_y;
fl := Flow(evalDG(b1*t1 + b2*t2 + b3*t3), s);
tb := Transformation(N, N, ApplyTransformation(fl, [x = x, y = y, theta = theta]));

t0 := evalDG(theta*D_x + (theta^2/2 - x^2/2)*D_y - x*D_theta);
f2 := simplify(Flow(t0, s));
t2 := simplify(Transformation(N, N, ApplyTransformation(f2, [x = x, y = y, theta = theta])));

ChangeFrame(N);
inf_lagr := InfinitesimalSymmetriesOfGeometricObjectFields([[n1], [n2]], output = "list");
alg_lagr := LieAlgebraData(inf_lagr, al);
DGsetup(alg_lagr);
Query("Semisimple");
Query([e1, e2, e3, e4, e5, e6, e7, e8], "SplitForm");

rep := LieAlgebras:-Adjoint(alg_lagr);
a1 := SubMatrix(rep[4], [2, 3], [2, 3]);
a2 := SubMatrix(rep[5], [2, 3], [2, 3]);

psol2mod := solve({rhs(poc2[1]), rhs(poc2[2]), rhs(poc2[3])}, {_C1, _C4, _C5, _C6});
sol20mod := simplify(subs(psol2mod, sol2));
t6 := evalDG(y/2*D_x - theta^2/2*D_theta);
f6 := simplify(Flow(t6, s));
p6 := simplify(PushPullTensor(f6, InverseTransformation(f6), [n1, n2]));
evalDG(h1(t)*p6[1] + h2(t)*p6[2]);
GC := subs(x = x(t), y = y(t), theta = theta(t), GetComponents(%, [D_x, D_y, D_theta]));
sys6 := {op(sys_vert), diff(theta(t), t) = GC[3], diff(x(t), t) = GC[1], diff(y(t), t) = GC[2]};
at6 := simplify(ApplyTransformation(f6, [rhs(sol20mod[2]), rhs(sol20mod[3]), rhs(sol20mod[1])]));
test := {op(sol_vert[2]), theta(t) = at6[3], x(t) = at6[1], y(t) = at6[2]};
pdetest(test, sys6);

Bra := [[E2, E3] = E1, [E4, E1] = 2*E1, [E4, E2] = E2, [E4, E3] = E3, [E5, E1] = 0, [E5, E2] = E2, [E5, E3] = -E3];
Alg := LieAlgebraData(Bra, [E1, E2, E3, E4, E5], A, grading = [-2, -1, -1, 0, 0]);
DGsetup(Alg);
PA := TanakaProlongation(A, 1, AlgPA);

Bra := [[E2, E3] = E1, [E4, E1] = 2*E1, [E4, E2] = E2, [E4, E3] = E3, [E5, E1] = 0, [E5, E2] = E2, [E5, E3] = -E3];
Alg := LieAlgebraData(Bra, [E1, E2, E3, E4, E5], A, grading = [-2, -1, -1, 0, 0]);
DGsetup(Alg);
PA := TanakaProlongation(A, 2, AlgPA);

\end{verbatim}

\end{document}